\documentclass{amsart}
\usepackage{amsmath, amsfonts, amssymb,amsbsy,eucal,mathrsfs}
\usepackage[all]{xy}
\usepackage{pstricks}
\usepackage{pst-plot}
\usepackage{graphicx}
\usepackage[utf8]{inputenc}
\usepackage[T1]{fontenc}
\usepackage{tikz}

\newtheorem{thm}{Theorem}[subsection] 
\newtheorem{lemma}[thm]{Lemma}
\newtheorem{prop}[thm]{Proposition}
\newtheorem{cor}[thm]{Corollary}

\theoremstyle{definition}
\newtheorem{remark}[thm]{Remark}
\newtheorem{example}[thm]{Example}
\newtheorem{defn}[thm]{Definition}

\renewenvironment{proof}{{\flushleft \it Proof.}}{\hfill $\square$ \vspace{2mm}}

\newenvironment{proofofthm}{{\flushleft \it Proof of Theorem \ref{thm-mds1}.}}{\hfill $\square$ \vspace{2mm}}

\DeclareMathOperator{\Hom}{Hom}

\DeclareMathOperator{\codim}{codim}

\newcommand{\R}{{\mathbb R}}
\newcommand{\Z}{{\mathbb Z}}
\newcommand{\Q}{{\mathbb Q}}

\newcommand{\cF}{{\mathcal F}}
\newcommand{\cG}{{\mathcal G}}

\newcommand{\cO}{{\mathcal O}}
\newcommand{\cI}{{\mathcal I}}

\newcommand{\pic}[2]{\includegraphics[scale=#1]{#2}}

\newcommand{\ignore}[1]{}

\def\pic{{\rm Pic}}
\def\Supp{{\rm Supp}}

\def \co {{{\mathcal O}}}
\def\N{{\mathbb N}}

\def \ka {{\mathfrak{k}}}

\def \spec {{{\rm Spec}}}
\def \proj {{{\rm Proj}}}
\def \cF {{{\mathcal F}}}
\def \cox {{{{\mathcal C}_Y^\vee(X)}}}

\def \X {{{\mathbb{X}}}}
\def \a {{\alpha}}

\def \Hom {{{\rm Hom}}}

\def\aut{{\rm Aut}}

\def\s{\sigma}
  
\newcommand{\G}{\mathbb{G}}

\def \Xt {{\widetilde{X}}}
\def \Yt {{\widetilde{Y}}}

\newcommand{\bF}{\mathbb{F}}
\newcommand{\p}{\mathbb{P}}

\renewcommand{\a}{{\alpha}}

\newcommand{\rk}{{\rm rk}}

\def \div {{{\textrm{div}}}}

\newcommand{\F}{{\mathbb{F}}}

\newcommand{\scal}[1]{\langle #1 \rangle}

\def \cX {{{\mathfrak{X}}}}

\def \cC {{\mathcal C}}

\def \char {{{\textnormal{char}}}}
\newcommand{\g}{\mathfrak g}

\newcommand{\h}{\mathfrak h}

\newcommand{\supp}{{\rm Supp}}

\begin{document}

\title{On the geometry of spherical varieties}

\date{July 08, 2012}

\author{Nicolas Perrin}
\address{Mathematisches Institut, Heinrich-Heine-Universit{\"a}t,
  40204 D{\"u}sseldorf, Germany.}
\email{perrin@math.uni-duesseldorf.de}


\begin{abstract}
This is a survey article on the geometry of spherical varieties.
\end{abstract}

\maketitle

\vskip -0.6 cm

\markboth{N.~PERRIN}
{\uppercase{On the geometry of spherical varieties}}

\setcounter{tocdepth}{2}
\tableofcontents

\vskip -1.4 cm

\

%%%%%%%%%%%%%%%%%%%%%%%%%%%%%%%%%%%%%%%%%%%%%%%%%%%%%%%%%%%%%%%%%%%%%%%%%

\section{Introduction}\label{sec:intro}

This paper is a survey on the geometry of spherical varieties.
The problem of classifying algebraic varieties is very natural. It
can be decomposed in two: classify algebraic varieties modulo
birational transformations and describe all the elements in a
birational class. These two problems are in general extremely hard
and it is natural to impose conditions on the varieties we want to
classify. The most natural conditions to impose are restrictions on
the singularities. It is also natural to impose additional
structures. In particular if $G$ is a reductive\footnote{Assuming
that $G$ is reductive has several implications which make this
choice important: finite generation properties, good representation
theory.} algebraic group we may consider $G$-varieties \emph{i.e.}
varieties with an action of $G$. Even with such restrictions the
classification problem stays very hard.

The action of $G$ however gives a measure of the complexity of the
varieties: the simplest $G$-varieties being the varieties with
finitely many orbits. From this perspective, the simplest class of
varieties we can consider is the following class.

\begin{defn}
A $G$-spherical variety is a normal $G$-variety such that all its
$G$-equivariant birational models have finitely many $G$-orbits.
\end{defn}

This definition is not the usual definition of spherical varieties
but we want to emphasize this more intrinsic and geometric
definition. We will give equivalent definitions in Section
\ref{section-char}. Spherical varieties generalise several classes
of ubiquitous algebraic varieties: toric varieties, projective
rational homogeneous spaces and symmetric varieties.

\medskip

As one can expect from their definition, spherical varieties have a
nice equivariant birational behaviour. In particular any projective
spherical variety is a Mori Dream Space (see Theorem \ref{thm-mds1})
and therefore Mori program runs perfectly over any characteristic
for spherical varieties (see Section \ref{section-mori}).
Furthermore the above classification program is almost completed:
the Luna-Vust theory (see Section \ref{section-class}) describes all
the birational models of a given spherical variety using colored
fans
--- a generalisations of fans used to classify toric varieties.
Furthermore, Luna proved that the description of all birational
classes is equivalent to the classification of wonderful varieties
and gave a conjectural classification. Two solutions for this
classification have been recently proposed (see Section
\ref{subsection-wonderful}).

\medskip

Spherical varieties have many other nice geometric properties. The
characteristics of the action of the group $G$ imply for example
that the Chow groups are equal to the invariant Chow groups and are
finitely generated. For smooth spherical varieties, they agree with
the homology groups (see Section \ref{section-FMSS}). The group
action also gives a strong control on the local structure of
spherical varieties. In particular all spherical varieties are
rational. In characteristic $0$, they have rational singularities
and therefore are Cohen-Macaulay. This in turn implies vanishing
results for nef line bundles (see Section \ref{section-struct}). For
spherical varieties in positive characteristic coming from reduction
of spherical varieties in characteristic $0$, we have Frobenius
splitting properties and the same regularity and vanishing results
hold (see Section \ref{sec-split})

\medskip

The classification of spherical varieties in the same birational
class via colored fans has many geometric applications. For example,
as in the case of toric varieties, there is an explicit description
of the Picard group and the group of Weil divisors in terms of
convex geometry (see Section \ref{sec-div}). Among other
consequences this yields criteria of $\Q$-factoriality,
quasi-projectivity, properness or affiness for spherical varieties
(see Sections \ref{section-class} and \ref{sec-div}). For projective
varieties, the fan geometry can partially be translated into convex
properties of the moment polytope associated to the moment mapping (see
Subsection \ref{section-convex}). More generally, Okounkov
\cite{okounkov1} and latter Lazarsfeld and Musta\c t\u a
\cite{lazarsfeld-mustata} and Kaveh and Khovanski\u\i\
\cite{KK-annals} defined the so called Okounkov bodies. Many
geometric properties of the variety are encoded in the Okounkov
body. This takes a particularly simple form for spherical varieties
and leads in characteristic $0$ to flat deformations of spherical
varieties to toric varieties, to a smoothness criterion, to formulas
for computing the degree of line bundles and to an explicit
presentation of the subring of the cohomology ring generated by the
first Chern classes of line bundles (see Section \ref{sec-conv}).

\medskip

In this paper we tried to concentrate on the geometric properties of
spherical varieties and therefore avoided as much as possible
considering more particular classes of spherical varieties such as
toric varieties, horospherical varieties or symmetric varieties. We
nevertheless discuss the specificities of toroidal and sober
varieties in more details.

Toroidal varieties arise naturally from the colored fan
classification point of view since by definition their fan do not
have colors. For this reason they are tightly connected to toric
varieties as their same suggests. Toroidal varieties are also natural
from the log-geometric point of view since the smooth toroidal
varieties are exactly the log-homogeneous varieties\footnote{For a
linear algebraic group.} (see \cite{brion-log}, \cite{brion-log2}
for more on the log-geometric point of view). Furthermore, the study
of toroidal varieties leads to the description of a $B$-stable
canonical divisor and to explicit equivariant resolutions of
singularities in characteristic $0$ of any spherical variety (see
Section \ref{section-toro}).

Sober varieties are the spherical varieties having a complete
toroidal variety with a unique closed orbit in their birational
class. This completion is unique in the birational class. Thanks to
results of Luna \cite{luna}, the classification of all birational
classes of spherical varieties amounts to classifying these complete
toroidal varieties with a unique closed orbit which are smooth. Such
varieties are called wonderful varieties and a classification has
recently been proposed by Cupit-Foutou and by Bravi and Pezzini (see
Section \ref{section-sober}).

\medskip

Inspired by projective rational homogeneous spaces, the study of
closures of $B$-orbits\footnote{For $B$ a Borel subgroup. Spherical
varieties have finitely many $B$-orbits, see Theorem
\ref{theo-char}.} in spherical varieties is an active and rich
subject. We give a brief tour and describe in particular a graph
whose vertices are the closures of $B$-orbits. In some special cases
the properties of this graph imply normality or non-normality
results for orbit closures. We also present an application of these
techniques proving regularity results on multiplicity-free
subvarieties of projective rational homogeneous spaces (see Section
\ref{sec-B-orb}).

We end the paper with a short list of further examples of spherical
varieties.

\addtocontents{toc}{\protect\setcounter{tocdepth}{1}}
\subsection*{Notation (groups)} In this paper $G$ will be a connected
reductive algebraic group. We will denote by $H$ a closed spherical
subgroup \emph{i.e.} such that $G/H$ is a spherical variety. We
denote by $T$ a maximal torus of $G$, by $W$ the associated Weyl
group and by $R$ the corresponding root system. We denote by $B$ a
Borel subgroup of $G$ and by $R^+$ and $R^-$ the corresponding sets
of positive roots and negative roots. We denote by $U$ the unipotent
radical of $B$. For $\Gamma$ a group, we denote by $\cX(\Gamma)$ the
group of characters and by $[\Gamma,\Gamma]$ the derived group. For
$\lambda$ a dominant character of $T$, we denote by $V(\lambda)$ the
highest weight module of highest weight $\lambda$. For $V$ a
representation of a group $\Gamma$, we denote by $V^\Gamma$ the
subspace of invariant and by $V^{(\Gamma)}_\lambda$ the subspace of
semiinvariants of weight $\lambda\in\cX(\Gamma)$. The sum of all
these subspaces of semiinvariants is denoted by $V^{(\Gamma)}$.

\subsection*{Notation (varieties)}
We work over an algebraically closed field $k$. Varieties will be
irreducible and reduced. A $G$-variety is a variety together with an
action of the group $G$. Let $H$ be a spherical subgroup of $G$. The
quotient $G/H$ is called a spherical homogeneous space. Any
spherical variety $X$ which is $G$-birational to $G/H$ is called an
embedding of $G/H$. For $\Gamma$ an algebraic group and $\Gamma'$ a
closed subgroup acting on a variety $X$, if the quotient of
$\Gamma\times X$ under the action of $\Gamma'$ given by
$\gamma'\cdot(\gamma,x)=(\gamma\gamma',{\gamma'}^{-1}\cdot x)$
exists, we call it the contracted product and write
$\Gamma\times^{\Gamma'}X$. Such a quotient exists for example if $X$
has ample $\Gamma'$-linearisable line bundle or if the quotient
$\Gamma\to\Gamma/\Gamma'$ is locally trivial for Zarisky topology (e.g. if
$\Gamma'$ is a parabolic subgroup of $\Gamma$). The contracted
product is a locally trivial fibration (for the {\'e}tale topology)
with fibers isomorphic to $X$ over $\Gamma/\Gamma'$.

\addtocontents{toc}{\protect\setcounter{tocdepth}{1}}
\subsection*{Foreword} This is a survey article. I tried
to gather the most important results on spherical varieties with an
emphasize on the geometry. I also tried to give sketch of proofs of
the results when this was not too long as well as to give precise
references to all the results presented. The presentation does not
always follow a logical order: I sometimes use results coming later
in the text but the reader will easily check the consistency of the
proofs.

\subsection*{Acknowledgement} I want to thank Michel Brion for many discussions and email exchanges on preliminary versions of this text. Part of this work was written while I was a BJF at the HCM in Bonn. This is the occasion to acknowlegde the wonderful working conditions I had there.

%%%%%%%%%%%%%%%%%%%%%%%%%%%%%%%%%%%%%%%%%%%%%%%%%%%%%%%%%%%%%%%%%%%%%%%%%

\addtocontents{toc}{\protect\setcounter{tocdepth}{2}}
\section{First geometric properties}

%%%%%%%%%%%%%%%%%%%%%%%%%%%%%%%%%%%%%%%%%%%%%%%%%%%%%%%%%%%%%%%%%%%%%%%%%

\subsection{Characterisations of spherical varieties}
\label{section-char}

In this section, we give several equivalent definitions of spherical
varieties.

\begin{defn} Let $X$ be a $G$-variety and $V$ a $G$-module.

(\i) The complexity $c(X)$ of $X$ is the minimal codimension of a
$B$-orbit in $X$.

(\i\i) $V$ is multiplicity-free if for any $\lambda\in\cX(T)$
dominant, $\dim\Hom_G(V(\lambda),V)\leq1$.
\end{defn}

\begin{thm}
\label{theo-char}
Let $X$ be a normal $G$-variety. The following are equivalent.

(\i) $k(X)^B=k$.

(\i\i) $c(X)=0$

(\i\i\i) $X$ has finitely many $B$-orbits.

(\i v) $X$ is spherical.

\noindent If $X$ is quasi-projective, then these conditions are
equivalent to:

(v) For $\mathcal{L}$ a $G$-linearised line bundle, the $G$-module
$H^0(X,\mathcal{L})$ is multiplicity-free.
\end{thm}

\begin{proof}
{\it (\i)}$\Leftrightarrow${\it (\i\i)} by Rosenlicht's Theorem
\cite{rosentlicht}.

{\it (\i\i)}$\Rightarrow${\it (\i\i\i)} Let $Y$ be a closed
$B$-stable subvariety of $X$ we want to prove $c(Y)=0$. Using the
map $P\times^BY\to PY$ which is generically finite for $P$ a minimal
parabolic subgroup containing $B$ with $PY\neq Y$, we get that
$c(PY)=c(P\times^BY)$ and the latter is easily proved to be greater
than $c(Y)$: the fibration $P\times^BY\to P/B$ is locally trivial
with fiber $Y$ so that $c(P\times^BY)$ is equal to the minimum of
the codimension of the orbits of a subgroup of $B$ in $Y$. Since the
minimal parabolic subgroups containing $B$ span $G$, we may assume
that $Y$ is $G$-stable. By classical representation theoretic
arguments and Sumihiro's local structure Theorem on normal
$G$-varieties (see Theorem \ref{sumihiro}), there exists an open
$B$-stable affine subset $X_0$ meeting $Y$ non trivially such that
any element in $k[X_0\cap Y]^{(B)}$ can be lifted to an element in
$k[X_0]^{(B)}$ (up to taking some higher power in positive
characteristic, see also \cite[Theorem 1.3]{knop}). As $f\in k(Y)^B$
can be written $u/v$ with $u,v\in k[X_0\cap Y]^{(B)}$ with the same
weight, there exists $n>0$ and $u',v'\in k[X_0]^{(B)}$ such that
$u'\vert_Y=u^n$ and $v'\vert_Y=v^n$. We get $(u'/v')\vert_Y=f^n$ and
the transcendence degree of $k(X)^B$ is bigger than the
transcendence degree of $k(Y)^B$.

{\it (\i\i\i)}$\Rightarrow${\it (\i v)} Follows from {\it
  (\i\i)}$\Rightarrow${\it (\i\i\i)}: $c(\Xt)=0$ for any $G$-birational
model $\Xt$.

{\it (\i v)}$\Rightarrow${\it (v)} We may assume that $X=G/H$ and that
$\mathcal{L}$ is $G$-linearised. If the multiplicity-free assumption
fails, there exist independent vectors $v,w$ in $k[G]^{(H)}$ of
the same weight. One can check that the closure of
$G\cdot(u+v)$ in $\p(V)$ where $V$ is the $G$-module spanned by $u$
and $v$ has infinitely many orbits. Taking the closure of
$G\cdot(x,u+v)$ for a general $x\in X$ gives a $G$-birational model with
infinitely many orbits.

{\it (v)}$\Rightarrow${\it (\i)}. Let $f\in k(X)^B$, there exists a
$G$-linearised line bundle $\mathcal{L}$ and sections $\s,\s'$ in
$H^0(X,\mathcal{L})^{(B)}$ with the same weight such that
$f=\s/\s'$. These sections are colinear and $f$ is constant.
\end{proof}

\begin{cor}
Spherical varieties are rational.
\end{cor}

\begin{proof}
Let $G/H$ be a spherical homogeneous space, by the previous result,
there exists a Borel subgroup $B$ such that $B/B\cap H$ is dense in
$G/H$. It is therefore enough to prove that a quotient $B/K$ is
rational for any connected solvable group $B$ and any subgroup
scheme $K$. We proceed by induction on $(\dim B,\dim B/K)$ with the
lexicographical order. Pick $Z$ a connected one dimensional normal
subgroup of $B$. Such a subgroup always exists: if the unipotent
radical of $B$ is trivial then $B$ is a torus and pick for $Z$ any one
dimensional subgroup. Otherwise, pick for $Z$ a one dimensional
subgroup of the center of the unipotent radical $U$ of $B$ (this
center is non trivial and such a subgroup exist for example by
\cite[Theorem 10.6(2)]{borelb}). If $Z\subset K$ then quotienting by
$Z$ we conclude by induction on $\dim B$. Otherwise, we have a
fibration $B/K\to B/KZ$ obtained after quotienting by the action of
$Z$ via left multiplication (recall that $Z$ is normal and that $KZ$
is the subgroup generated by $K$ and $Z$). The fiber of this fibration
is $Z/(Z\cap K)$ which is a connected group of dimension 1. Since any
connected solvable group is special (see \cite[Proposition 14]{serre})
the fibration $B/K\to B/KZ$ is locally trivial and since $B/KZ$ is rational
(by induction on $\dim B/K$) the result follows.
\end{proof}

\subsubsection{Comments}
Spherical varieties are usually defined by condition {\it (\i)} or
{\it (\i\i)} of the previous result. Brion \cite{brion-def-sphe}, Vinberg
and Kimel'fel'd \cite{vinberg} and Popov \cite{popov} prove the
equivalence with condition {\it (\i\i\i)}. The equivalence with the last
condition was proved in \cite{vinberg}, \cite{brion-rep} and
\cite{bien}. We refer to \cite[Chapter 5]{timashev} for other
equivalent definitions of spherical varieties.

%%%%%%%%%%%%%%%%%%%%%%%%%%%%%%%%%%%%%%%%%%%%%%%%%%%%%%%%%%%%%%%%%%%%%%%%%

\subsection{Chow groups and homology}
\label{section-FMSS}

As a first step into the geometry of spherical varieties, we
describe and compare Chow groups, homology groups and cohomology
groups of spherical varieties.

\vskip 0.3 cm

For a variety $X$, we denote by $Z_kX$ the free abelian group
generated by all $k$-dimensional closed subvarieties and by $R_kX$
the subgroup generated by divisor $[{\rm div}(f)]$ where $f\in
k(W)^\times$ with $W$ a $(k+1)$-dimensional subvariety in $X$. If
$X$ is a $B$-variety, we denote by $Z_k^BX$ the free abelian group
generated by the $B$-stable closed subvarieties of $X$ and by
$R_k^BX$ the subgroup generated by divisors of the form $[{\rm
div}(f)]$ where $f\in k(W)^{(B)}$ with $W$ a $(k+1)$-dimensional
subvariety in $X$. We denote by $A_kX$ and $A_k^BX$ the quotients
$Z_kX/R_kX$ and $Z_k^BX/R_k^BX$. We denote by $A_*X$ and $A_*^BX$
the direct sum of the groups $A_kX$ and $A_k^BX$ for all $k\geq0$.

\begin{thm}
\label{theo-chow}
Let $X$ be a spherical variety, then the canonical
homomorphism $A_*^BX\to A_*X$ is an isomorphism.
\end{thm}

\begin{proof}
This result is proved in two steps. First using a result of H.
Sumihiro \cite{sumihiro} (see also Theorem \ref{sumihiro}) there is
an equivariant Chow Lemma and by standard exact sequences for Chow
groups (see \cite[Section 1]{FMSS}) one can reduce the proof to the
case where $X$ is projective. This was proved by A. Hischowitz
\cite{hirscho}. The surjectivity of the map comes from the Borel
fixed point Theorem applied to the Chow variety: the class in the
Chow variety of any subvariety in $X$ has a $B$-fixed point in the
closure of its $B$-orbit. Applying similar arguments to the rational
curves in the Chow varieties proves that $R_k^BX$ generates $R_kX$.
\end{proof}

\begin{cor}
\label{cor-polehedral}
Let $X$ be a complete spherical variety.

(\i) The groups $A_kX$ are of finite type for all $k\geq0$.

(\i\i) The cone of effective cycles in $A_kX\otimes_\Z\Q$ is a
polyhedral convex cone generated by the classes of the closures of the
$B$-orbits.

(\i\i\i) Rational and algebraic equivalences coincide on $X$.

(\i v) Rational and numerical equivalences coincide for Cartier
divisors on $X$.
\end{cor}

\begin{proof}
{\it (\i)} follows from the previous result. {\it (\i\i)} follows
from the proof of the previous result: by Chow's Lemma we may assume
$X$ projective. The class of an effective cycle has a $B$-fixed
limit in the Chow variety giving the result. For {\it (\i\i\i)}
recall from \cite[Section 19.1.4]{fulton} that the group of
algebraically trivial cycles modulo rational equivalence is
divisible. By {\it (\i)} it is finitely generated therefore trivial.
Finally {\it (\i v)} follows from the fact that $\pic(X)$ is torsion
free (see Corollary \ref{pic-sans-tors}) and the fact that the group
of homologically trivial divisors modulo rational equivalence is a
torsion group \cite[Section 19.3.3]{fulton}.
\end{proof}

\begin{remark}
The Chow groups may in general have torsion even for divisors. For
example, the affine quadratic cone $X$ in $\mathbb{A}^3$ over a
smooth conic in $\p^2$  is toric singular with $A_1X=\Z/2\Z$ (see
\cite[Example 2.3]{fulton-sturmfels} for more examples).
\end{remark}

When the variety is smooth and complete, we get a stronger result.

\begin{cor}
\label{cor-chow-lisse}
If $X$ is a smooth complete spherical
variety, then the canonical homomorphism $A_*X\to H_*X$ is an
isomorphism.
\end{cor}

\begin{proof}
By easy considerations on the $B$-orbits, one can prove that for any
variety $Y$ with a trivial $B$-action the equivariant Kunneth map
$A_*^BX\otimes A_*^BY\to A_*^B(X\times Y)$ is an isomorphism (see
\cite[Lemma 3]{FMSS}). Using Theorem \ref{theo-chow} we get that the
Kunneth map $A_*X\otimes A_*Y\to A_*(X\times Y)$ is an isomorphism.
From there, classical arguments of G. Ellingsrud and S.A. Str\o mme
\cite{ellingsrud-stromme} imply the
result: the class of the diagonal $[\Delta]\in A_*(X\times
X)$ is in the image of $A_*X\otimes A_*X$. In particular it is of
the form $[\Delta]=\sum_i a_i\otimes b_i$ for some $a_i,b_i\in A_*X$
therefore any class $c$ is of the form $c=\sum_i(a_i\cdot c)b_i$.
This proves that the classes $b_i$ generate $A_*X$ and $H_*X$ and
that numerical and rational equivalences coincide.
\end{proof}

\begin{remark}
 If $X$ is not smooth this result fails: there are examples of
singular toric varieties with non trivial odd homology groups, see
\cite{McConnell} for more details.
\end{remark}

\subsubsection{Comments}
Theorem \ref{theo-chow} and Corollary \ref{cor-chow-lisse} are taken
from \cite{FMSS}. In that paper the authors proves that the
conclusion of Theorem \ref{theo-chow} holds (with almost the same
proof) for any scheme $X$ with an action of a connected solvable
group $B$. The first results in this direction were obtained for
projective varieties by Hirschowitz \cite{hirscho}. Corollary
\ref{cor-polehedral} is taken from \cite[Section 1.3]{brion-mori}.
Note that the above results remain true in positive characteristic
if we replace singular homology by {\'e}tale homology. Note also
that for $X$ a complete spherical variety, Brion \cite{brion-aus}
proved that the following equality holds:
$\Hom_\Z(\pic(X),\Z)=A_1(X)$. Recall the definition of the Chow
cohomology group $A^1(X)$ from \cite[Chapter 17]{fulton}. The
equality $A^1(X)=\Hom_\Z(A_1(X),\Z)$ also holds, see \cite{FMSS}
thus $\pic(X)=A^1(X)$.

%%%%%%%%%%%%%%%%%%%%%%%%%%%%%%%%%%%%%%%%%%%%%%%%%%%%%%%%%%%%%%%%%%%%%%%%%

\subsection{Local structure of spherical varieties}
\label{section-struct}

In this section we describe the local structure of spherical
varieties. More precisely, for $X$ a spherical variety, we exhibit
open coverings of $X$ with good properties with respect to the
action of the group. We deduce regularity properties and cohomology
vanishing results for spherical varieties in characteristic $0$.

\subsubsection{Local structure}
Let us first recall a result of Sumihiro \cite{sumihiro}.

\begin{thm}
\label{sumihiro} Let $X$ be a normal $G$-variety and let $Y$ be a
$G$-orbit.

(\i) There exists a quasi-projective $G$-invariant open subset of
$X$ containing $Y$.

(\i\i) If $X$ is quasi-projective, there exists a $G$-linearised
line bundle $\mathcal L$ inducing a $G$-equivariant embedding in
$\p(H^0(X,\mathcal{L})^\vee)$.
\end{thm}

Because of this result, to give a Local Structure Theorem, we may
assume that $X$ is quasi-projective equivariantly embedded in
$\p(V)$ where $V$ is a rational $G$-module. Let $\sigma\in
{V^\vee}^{(B)}$ a semi-invariant linear form and let $P$ be the
stabliser of $[\sigma]\in\p(V^\vee)$. The group $P$ is a parabolic
subgroup containing $B$. Denote by $L$ the Levi subgroup of $P$
containing $T$ and by $P_u$ the unipotent radical of $P$. Denote by
$X_\sigma$ the open subset of $X$ where $\sigma$ does not vanish.

\begin{thm}
\label{theo-loc-str}
(\i) $P_u$ acts properly on $X_\sigma$, the quotient $X_\sigma/P_u$
exists and the morphism $\pi:X_\sigma\to X_\sigma/P_u$ is affine.

(\i\i) There is a $T$-stable closed subvariety $Z$ of $X_\sigma$
such that the morphisms $P_u\times Z\to X_\sigma$ and $Z\to
X_\sigma/P_u$ are finite and surjective.

(\i\i\i) If $\char k=0$, then $Z$ can be chosen $L$-stable and such
that morphisms $P_u\times Z\to X_\sigma$ and $Z\to X_\sigma/P_u$ are
isomorphisms.
\end{thm}

\begin{proof}
It is enough to prove the result for $X=\p(V)$. Representation
theory of $G$ gives a $B$-semiinvariant vector $v\in V^{(B)}$ such
that $\sigma(v)=1$. Modifying $V$ we may even assume that for any
regular dominant cocharacter $\varpi^\vee:\G_m\to T$, the weight
$\lambda_v$ of $v$ has value $\scal{\varpi^\vee,\lambda_v}$ larger
than all the other weights of $V$. Taking a $T$-stable complement
$S$ of $kv$ in $V$, we set $Z=\p(S\oplus kv)$. Our assumptions imply
that the morphism $P_u\times Z\to X_\sigma$ is $\G_m$-equivariant
(where $\G_m$ acts via $\varpi^\vee$), that the only fixed points
are $([v],1)$ and $[v]$, and that the fiber of $[v]$ is $([v],1)$.
This implies that the map is finite and surjective.

Since $P_u\times Z\to X_\sigma$ is finite, it is proper and we get
that $k[Z]=k[Z\times P_u]^{P_u}$ is a finite $k[X_\sigma]$-module
proving the finiteness of the latter. This implies the existence
of the quotient. The theorem follows from this. For $\char k=0$ choose
$S$ to be a $L$-stable complement.
\end{proof}

\begin{remark}
In characteristic $0$, the variety $Z$ is $L$-spherical.
\end{remark}

\subsubsection{Application to singularities and vanishing}
\label{section-sing-van}
 In this subsection we assume $\char k=0$.

\begin{cor}
\label{cor-sing-rat}
Let $X$ be a spherical variety, then any $G$-stable subvariety has
rational singularities, in particular it is Cohen-Macaulay.
\end{cor}

\begin{proof}
By the Local Structure Theorem, we may assume $X$ to be affine. The affine
case is proved by deformation see Section \ref{sec-def} and
Corollary \ref{cor-sing-rat-def}.
\end{proof}

\begin{thm}
\label{thm-vanish}
Let $\varphi:X\to X'$ be a proper $G$-equivariant morphism between
spherical varieties and let $\mathcal{L}$ be a globally generated line
bundle on $X$. Then $R^i\varphi_*\mathcal{L}=0$ for all $i>0$.
\end{thm}

\begin{proof}
By the Local Structure Theorem, we may assume that $X'$ is affine. Then
one restricts to the case where $\mathcal{L}$ is trivial. Indeed,
let $L^\vee$ be the total space of the dual line bundle
$\mathcal{L}^\vee$. Since $\mathcal{L}$ is globally generated, we
have a proper morphism $\pi:L^\vee\to
Y=\spec\left(\oplus_nH^0(X,\mathcal{L}^n)\right)$. Up to replacing
$Y$ by its normalisation, these varieties are
$G\times\G_m$-spherical varieties.
Since $L^\vee\to X$
is affine it is enough to prove the result for $\pi$ and
$\co_{L^\vee}$ so that we can assume that $\mathcal{L}=\cO_X$. By an
application of Sumihiro's Theorem we may assume that the morphism is
projective.

We conclude by writing $X=\proj(A)$ for some normal graded algebra
$A=\oplus_nA_n$ such that $A$ is generated by $A_0$ and $A_1$ and
$X'=\spec(A_0)$. On $X=\proj(A)$ we have the line bundle $\cO_X(1)$.
and a commutative diagram
$$\xymatrix{Z=\spec\left( \oplus_n\cO_X(n)\right)\ar[r]^-\pi\ar[d]_\psi & \spec A\ar[d]\\
X\ar[r]^\varphi & X'\\}$$ where $\pi$ is proper and
\emph{birational}. The varieties on the top line are spherical for
$G\times\G_m$ thus have rational singularities. We get the vanishing
$R^i\pi_*\cO_Z=0$ for $i>0$ and thus the vanishing $H^i(Z,\cO_Z)=0$.
The result follows since $\psi$ is affine.
\end{proof}

\begin{cor}
Let $X$ be a spherical variety proper over an affine and let
$\mathcal{L}$ be a globally generated line bundle, then
$H^i(X,\mathcal{L})=0$ for $i>0$.
\end{cor}

The previous result is not true without the assumption ``proper over
an affine'' as the example of $\mathbb{A}^2\setminus\{0\}$ shows:
$H^1(\mathbb{A}^2\setminus\{0\},\cO)\neq0$. The next result is a
generalisation of both Kodaira and Kawamata-Viehweg vanishings. In
the next result, $\kappa(\mathcal{L})$ is the
  Iitaka dimension of $\mathcal{L}$ \emph{i.e.} the dimension of
  $\proj\oplus_nH^0(X,\mathcal{L}^n)$.

\begin{thm}
Let $\mathcal{L}$ be a globally generated line bundle on a complete
spherical variety $X$, then $H^i(X,\mathcal{L}^{-1})=0$ for
$i\neq\kappa(\mathcal{L})$.
\end{thm}

\begin{proof}
As in the proof of the previous Theorem, let $L^\vee$ be the total
space of the dual line bundle $\mathcal{L}^\vee$. We have a proper
morphism $\pi:L^\vee\to
Y=\spec\left(\oplus_nH^0(X,\mathcal{L}^n)\right)$ and up to replacing
$Y$ by its normalisation, these varieties are $G\times\G_m$-spherical
varieties. Therefore we have $R^i\pi_*\cO_{L^\vee}=0$ for $i>0$ and by
a result of Kempf \cite[Theorem 5]{kempf} we get $R^{\dim
  X-\kappa(\mathcal{L})}\pi_*\omega_{L^\vee}=\omega_Y$ while
$R^i\pi_*\omega_{L^\vee}=0$ for $i\neq\dim
X-\kappa(\mathcal{L})$. Since $\psi:L^\vee\to X$ is affine and
$\omega_{L^\vee}=\psi^*(\omega_X\otimes\mathcal{L})$, the result
follows by projection formula and Serre duality (since $X$
has rational singularities, it is Cohen-Macaulay).
\end{proof}

\begin{remark}
In our setting, Kawamata-Viehweg vanishing Theorem would imply that
for $\mathcal{L}$ nef we have $H^i(X,\mathcal{L}^{-1})=0$ for
$i<\kappa(\mathcal{L})$. Note that on spherical varieties nef and
globally generated line bundles agree (see Corollary \ref{nef=gg}).
\end{remark}

\subsubsection{Comments}

The Local Structure Theorem was first proved in by Brion, Luna and
Vust \cite{BLV} over a field of characteristic $0$. It holds in a
weaker form for any $G$-variety. The positive characteristic version
of this result was proved by Knop in \cite{knop-bewert}. For some
specific spherical varieties, the Local Structure Theorem as stated
in characteristic $0$ holds true in positive characteristic. This is
the case for toroidal embeddings of the group $G$ seen as a $G\times
G$-spherical variety (see \cite{strickland}). This also holds for
toroidal embeddings of symmetric varieties in characteristic
different from $2$ by results of \cite{deconcini-springer}. Some
extensions of these results are obtained by Tange in \cite{tange}.

The regularity and vanishing results of Section \ref{section-sing-van}
are taken from \cite{brion-borel-weil}. We will recover these results
via Frobenius splitting techniques in Section \ref{sec-split}.

%%%%%%%%%%%%%%%%%%%%%%%%%%%%%%%%%%%%%%%%%%%%%%%%%%%%%%%%%%%%%%%%%%%%%%%%%

\section{Fans and geometry}

%%%%%%%%%%%%%%%%%%%%%%%%%%%%%%%%%%%%%%%%%%%%%%%%%%%%%%%%%%%%%%%%%%%%%%%%%

\subsection{Classification of embeddings}
\label{section-class}

Spherical varieties admit a simple combinatorial description of a given
equivariant birational class. This classification of all embeddings
extends the classification of toric varieties by fans (see
\cite{fulton-toric}, \cite{Oda}).

Note that there always exists a smallest element in a $G$-birational
class: the dense $G$-orbit $G/H$.

\begin{defn} Let $G/H$ a homogeneous spherical variety.

(\i) The set of $G$-invariant valuations\footnote{A valuation $\nu$ is
called invariant if $\nu(g\cdot f)=\nu(f)$ for all $g\in G$ and all
$f\in k(G/H)$.} of $k(G/H)$ is denoted by ${\mathcal V}(G/H)$.

(\i\i) The weight lattice of $G/H$ is $\X(G/H)=\{\chi\in \cX(T)\ /\
k(G/H)^{(B)}_\chi\neq0\}$.
It is a free abelian group of finite rank. Its rank is the rank of the
spherical variety. Write $\X^\vee(G/H)$ for the dual lattice and
$\X_\Q(G/H)$ and $\X_\Q^\vee(G/H)$ for their tensor product with
$\Q$.

(\i\i\i) The set of colors $\mathcal{D}(G/H)$ is the set of
$B$-stable divisors of $G/H$.
\end{defn}

When there is no possible confusion, we will write $\X$, $\X^\vee$,
${\mathcal V}$ and ${\mathcal D}$ for $\X(G/H)$, $\X^\vee(G/H)$,
${\mathcal V}(G/H)$ and ${\mathcal D}(G/H)$. Any valuation $\nu$
induces a homomorphism $\rho_\nu:\X \to\Q$ defined by
$\rho_\nu(\lambda)=\nu(f)$ with $f\in k(G/H)^{(B)}_\lambda$. This is
well defined since
$k(G/H)^{(B)}_\lambda$ is one-dimensional. For any $D\in\mathcal{D}$
there is an associated valuation $\nu_D$. We therefore have a map
$$\rho:{\mathcal V}\cup{\mathcal{D}}\to\X^\vee_\Q.$$
It is injective on $\mathcal{V}$ (see \cite[Corollary 1.8]{knop})
but not on $\mathcal{D}$ in general.

\subsubsection{Simple spherical varieties}

Let $X$ be an embedding of $G/H$ and let $Y$ be a $G$-orbit in $X$.
The subset $X_{Y,G}=\{x\in X\ /\ Y\subset\overline{Gx}\}$ is
$G$-stable, open in $X$. The $G$-orbit $Y$ is the unique closed
$G$-orbit of $X_{Y,G}$. In particular $X$ can be covered by
$G$-stable open subsets with a unique closed orbit.

\begin{defn}
  A spherical variety with a unique closed orbit is called simple.
\end{defn}

By the above discussion, spherical varieties are covered by simple
spherical varieties. Note also that Sumihiro's Theorem \ref{sumihiro}
implies that simple spherical varieties are quasi-projective.

\begin{defn}
Let $X$ be a spherical variety and $Y$ be a $G$-orbit in $X$.

(\i) Define $\mathcal{D}_Y(X)=\{D\in \mathcal{D}\ /\ Y\subset \overline{D}\}$.

(\i\i) Define $\mathcal{D}(X)=\cup_Y\mathcal{D}_Y(X)$ the set of
colors of $X$.

(\i\i\i) Define $\mathcal{V}(X)$ the set of $G$-stable irreducible divisors of $X$.

(\i v) Define $\mathcal{V}_Y(X)=\{D\in \mathcal{V}(X)\ /\ Y\subset
\overline{D}\}$.
\end{defn}

\begin{thm}
\label{theo-char-simple}
  A simple spherical embedding $X$ of $G/H$ with closed orbit $Y$ is
  completely determined by the pair $(\mathcal{V}_Y(X),\mathcal{D}_Y(X))$.
\end{thm}

\begin{proof}
First the datum $\mathcal{D}_Y(X)$ determines two canonical
$B$-stable open subsets $X_{Y,B}$ and $X_0$ of $X$ and $G/H$
respectively by setting
$$X_{Y,B}=X\setminus\bigcup_{D\in \mathcal{D}\setminus
\mathcal{D}_Y(X)}\overline{D} \ \ \textrm{and}\ \
X_0=G/H\setminus\bigcup_{D\in \mathcal{D}\setminus
\mathcal{D}_Y(X)}D.$$ In particular $X_0$ can be recovered from
$\mathcal{D}_Y(X)$. Looking at the ring of functions we have
$k[X_{Y,B}]= \{f\in k[X_0]\ /\ \nu(f)\geq0\ \textrm{for all
$\nu\in\mathcal{V}_Y(X)$}\}$. Finally we recover $X$ from $X_{Y,B}$:
since $X$ is simple $X=GX_{Y,B}$.
\end{proof}

Let $X$ be a spherical variety and $Y$ a $G$-orbit.

\begin{defn}
Define $\cox\subset\X^\vee_\Q$ as the cone generated by
$\mathcal{V}_Y(X)$ and $\rho(\mathcal{D}_Y(X))$.
\end{defn}

\begin{defn} Let $G/H$ be a homogeneous spherical variety.

(\i) A colored cone for $G/H$, is a pair $(\cC,\cF)$ with $\cC\subset
\X^\vee_\Q$ and $\cF\subset \mathcal{D}(G/H)$ having the following properties.
\begin{itemize}
\item[(CC1)] $\cC$ is a cone generated by $\rho(\cF)$ and finitely many
elements of ${\mathcal V}$.
\item[(CC2)] The intersection $\cC^\circ\cap{\mathcal V}$ is non
  empty\footnote{Here $\cC^\circ$ denotes the relative interior of
    $\cC$.}.
\end{itemize}
The colored cone $(\cC,\cF)$ is called strictly convex if the
following condition holds.
\begin{itemize}
\item[(SCC)] The cone $\cC$ is strictly convex and $0\not\in\rho(\cF)$.
\end{itemize}

(\i\i) A colored face of $(\cC,\cF)$ is a pair $(\cC_0,\cF_0)$
satisfying is the following conditions.
\begin{itemize}
\item[(a)] The set $\cC_0$ is a face of the cone $\cC$.
\item[(b)] The intersection $\cC_0^\circ\cap {\mathcal V}$ is non empty.
\item[(c)] We have the equality $\cF_0=\cF\cap\cC_0$.
\end{itemize}
\end{defn}

\begin{thm}Let $G/H$ be a homogeneous spherical variety.
\label{theo-simple}

(\i) The map $X\mapsto (\cox,\mathcal{D}_Y(X))$ is a bijection
between the isomorphism classes of simple spherical embeddings $X$
of $G/H$ with closed orbit $Y$ and strictly convex colored cones.

(\i\i) Let $X$ be a spherical embedding of $G/H$ and let $Y$ be an
  orbit. Then there is a bijection $Z\mapsto (\cC_Z^\vee(X),\mathcal{D}_Z(X))$
  between the set of $G$-orbits in $X$ such that $\overline{Z}\supset
  Y$ and the set of faces of $(\cox,\mathcal{D}_Y(X))$.
\end{thm}

\begin{proof}
The main point in the proof is the existence of a spherical
embedding with a given colored cone (uniqueness follows from Theorem
\ref{theo-char-simple}). We sketch the construction of a spherical
variety $X$ with colored cone $(\cC,\cF)$ and refer to \cite[Theorem
3.1]{knop} for a complete proof. Condition (CC1) allows to construct
elements $(f_i)_{i\in[0,n]}\in k[G]$ which are $B\times
H$-semiinvariants with the same $H$-character and such that the open
subset $X_0$ defined in the proof of Theorem \ref{theo-char-simple}
is given by $X_0=D(f_0)$. Write $W$ for the $G$-span of the
$(f_i)_{i\in[0,n]}$. These elements define a $G$-equivariant
morphism $\varphi:G/H\to\p(W^\vee)$. Condition (SCC) implies that
this morphism has finite fibers and condition (CC2) implies that the
Stein factorisation $G/H\to X\to \p(W^\vee)$ yields a simple
spherical variety $X$ with cone $(\cC,\cF)$.
\end{proof}

\subsubsection{General case}

Since spherical varieties have an open covering of simple spherical varieties,
we only have to \emph{glue} together the simple pieces to get the general classification.
For this we need the following notions.

\begin{defn} A colored fan $\bF$ in $\X^\vee_\Q$ is a finite collection of colored cones
  $(\cC,\cF)$ satisfying the properties.

(CF1) Every colored face of a colored cone $(\cC,\cF)$ of $\F$ is in
$\F$.

(CF2) For any $\nu\in{\mathcal V}$ there is at most one colored cone
$(\cC,\cF)\in\bF$ with $\nu\in\cC^\circ$.

A colored cone is called strictly convex if any colored cone in $\F$
is strictly convex. This is equivalent to $(0,\emptyset)\in\F$.
\end{defn}

\begin{defn}
  Let $X$ be an embedding of $G/H$ we define $\F(X)$ to be the set of
  all colored cones $(\cox,\mathcal{D}_Y(X))$ for $Y$ a $G$-orbit in $X$.
\end{defn}

\begin{thm}
\label{thm-class-fan}
  The map $X\mapsto \F(X)$ is a bijection between isomorphism classes
  of embeddings and strictly colored fans.
\end{thm}

\begin{proof}
This is a consequence of Theorem \ref{theo-simple}. The simple
models given by the cones glue together along the smaller simple
spherical embedding given by colored faces. One then has to check
that the variety obtained this way is separated. This is done thanks
to the valuative criterion of separatedness. We refer to
\cite[Theorem 3.3]{knop} for the proof of this fact.
\end{proof}

\subsubsection{Morphisms}

Let $\varphi:G/H\to G/H'$ be a surjective $G$-morphism
of homogeneous spherical varieties. We have a
a surjection $\varphi_*:\xymatrix{\X^\vee(G/H')\ar[r]&
\X^\vee(G/H)}$ and an equality $\varphi_*({\mathcal
V}(G/H))={\mathcal V}(G/H')$.

\begin{defn}
  (\i) Denote by $\cF_\varphi$ the set of $D\in \mathcal{D}(G/H)$ such
  that $\varphi$ maps $D$ surjectively to $G/H'$. Let $\cF_\varphi^c$
  be its complement in $\mathcal{D}(G/H)$.

(\i\i) Denote again by $\varphi_*:\cF_\varphi^c\to \mathcal{D}(G/H')$ the map
defined by $D\mapsto\varphi(D)$.
\end{defn}

\begin{defn}
  (\i) Let $(\cC,\cF)$ and $(\cC',\cF')$ be colored cones of $G/H$ and
  $G/H'$ respectively. We say that $(\cC,\cF)$ maps to $(\cC',\cF')$
  if the following conditions holds.

\begin{itemize}
\item[(CM1)] We have the inclusion $\varphi_*(\cC)\subset \cC'$.
\item[(CM2)] We have the inclusion $\varphi_*(\cF\setminus\cF_\varphi)\subset
\cF'$.
\end{itemize}

(\i\i) Let $\F$ and $\F'$ be colored fans of embeddings of $G/H$ and
$G/H'$ respectively. The colored fan $\F$ maps to $\F'$ if every
colored cone of $\F$ is mapped to a colored cone of $\F'$.
\end{defn}

Let $\varphi:G/H\to G/H'$ be a surjective morphism between spherical
homogeneous spaces. Let $X$ and $X'$ be embeddings of $G/H$ and
$G/H'$ respectively.

\begin{thm}
$\varphi$ extends to $X\to X'$  if and only if $\F(X)$ maps to
$\F(X')$.
\end{thm}

\begin{proof}
It is enough to assume that $X$ and $X'$are simple embeddings.
Consider the open subsets $X_0$, $X_{Y,B}$ in $X$ and $X'_0$,
$X'_{Y,B}$ in $X'$ as defined in the proof of Theorem
\ref{theo-char-simple}. The condition (CM2) is equivalent to the
fact that $\varphi$ restricts to $\varphi_0:X_0\to X'_0$. Using the
characterisation of $X_{Y,B}$ and $X'_{Y,B}$ given in Theorem
\ref{theo-char-simple}, condition (CM1) is equivalent to the
existence of a morphism $\varphi_{Y,B}:X_{Y,B}\to X'_{Y,B}$
extending $\varphi_0$.
\end{proof}

\begin{defn}Let $\F$ be a colored fan.
Define the support of $\F$ by
$$\supp(\F)=\mathcal{V}(G/H)
\bigcap
\bigcup_{(\cC,\cF)\in\F}\cC.
$$
\end{defn}

Let $\varphi:X\to X'$ be a dominant morphism extending a surjective
morphism $G/H\to G/H'$ between spherical embeddings.

\begin{thm}
$\varphi$ is proper if and only if
$\supp(\F(X))=\varphi_*^{-1}(\supp(\F(X')))$.

In particular $X$ is complete if and only if $\supp(\F(X))={\mathcal
V}(G/H)$.
\end{thm}

\begin{proof}
This is a classical application of the valuative criterion of
properness. Note that the inclusion
$\supp(\F(X))\subset\varphi_*^{-1}(\supp(\F(X')))$ is always
satisfied.
\end{proof}

\begin{remark}
Non strictly convex colored fans classify dominant $G$-morphisms
from $G/H$ with irreducible and reduced fibers, see \cite[Theorem
4.5]{knop}.
\end{remark}

\subsubsection{Structure of $G$-orbits}
Let $X$ be a spherical $G$-variety. Let us first note the following corollary of Theorems \ref{theo-simple} and \ref{thm-class-fan}.

\begin{cor}
There is a decreasing bijection between the set of $G$-orbits in $X$ and the set of colored faces of the fan of $X$.
\end{cor}

\begin{thm}
\label{theo-rk}
Let $Y$ be a $G$-orbit, then $\rk(Y)=\rk(X)-\dim\cox$.
\end{thm}

\begin{proof}
Let $\Lambda$ be the set of characters vanishing on $\cox$.
Then restriction of $B$-semiinvariant functions gives
the inclusion $\Lambda\subset\X(Y)$. But, up to $p$-torsion, $B$-semiinvariant
rational functions on $Y$ can be extended to $B$-semiinvariant
rational functions on $X$ (see for example \cite[Theorem 1.3b]{knop}).
\end{proof}

\begin{thm}
Assume $\char k=0$. Any closed $G$-stable subvariety $Y$ of $X$ is a spherical variety.
\end{thm}

\begin{proof}
Normality follows from Corollary \ref{cor-sing-rat} or Corollary
\ref{cor-sing-rat-def}. Theorem \ref{theo-char} implies that $X$ and
thus $Y$ has finitely many $B$-orbits. Thus $Y$ is spherical.
\end{proof}

\subsubsection{Comments}
\label{comments-class} The classification of embeddings, also called
\emph{Luna-Vust theory} was first studied in a more general setting
by Luna and Vust in \cite{LV}. For spherical varieties \emph{i.e.}
varieties of complexity $0$, this was latter simplified and extended
to any characteristic by Knop in \cite{knop}. We follow the
presentation of Knop \cite{knop}. This classification of spherical
embeddings can be partially generalised to varieties of complexity
$1$, see \cite{knop-bewert}, \cite{timashev}, \cite{timashev1} and
\cite{timashev2}.

To complete the classification of spherical varieties, one needs to
classify the homogeneous spherical spaces \emph{i.e.} the spherical
subgroups of a reductive group $G$. Luna \cite{luna} conjectured
that these subgroups should be classified by combinatorial objects
called \emph{spherical
systems}. In \cite{losev2}, Losev proves that there is at most one
spherical subgroup for each spherical system. There are two recent
propositions of a complete proof of the conjecture \emph{i.e.} the
existence of a spherical subgroup for any spherical system. The
first approach it via combinatoric and follows Luna's proof of type
$A$ in \cite{luna}. Type $D$ is solved in \cite{bravi-pezz1} and
type $E$ in \cite{braviE}. In the papers \cite{bravi-pezza},
\cite{bravi-pezzb}, \cite{bravi-pezzc}, Bravi and Pezzini propose a
general solution. A more geometric approach via invariant Hilbert
schemes as been proposed in \cite{cupit} by Cupit-Foutou. See
Subsection \ref{subsection-wonderful}, \cite[Section
4.5]{brion-hilbert} and \cite[Section 30.11]{timashev} for more
details.

Several partial classifications have been achieved. For example,
Akhiezer in \cite{akhiezer} classifies the spherical varieties of
rank one while Cupit-Foutou \cite{cupit-2orbits} classifies the
varieties with two orbits. In \cite{ruzzi1}, \cite{ruzzi2}, Ruzzi
gives a complete classification of smooth symmetric varieties with
Picard number one. In \cite{pasquier}, Pasquier classifies the
smooth horospherical varieties with Picard number one.

%%%%%%%%%%%%%%%%%%%%%%%%%%%%%%%%%%%%%%%%%%%%%%%%%%%%%%%%%%%%%%%%%%%%%%%%%%%%%%%

\subsection{Divisors}
\label{sec-div}

In this section, we describe the Picard group of spherical varieties
in terms of their fans. We also describe ample, nef and globally
generated line bundles. Comparing with the Weil divisors described
in Section \ref{section-FMSS} this leads to a characterisation of
$\Q$-factorial spherical varieties. We also give characterisations
of quasi-projective and affine spherical varieties.

\subsubsection{The Picard group}

Recall from Theorem \ref{theo-chow} that any Weil divisor $\delta$ can be written in the form
$\delta=\sum_Dn_D[D]$ where the sum runs over all $B$-stable irreducible divisors in $X$.

\begin{thm}
\label{prop-cartier}
  The divisor $\delta$ is Cartier if and only if for any $G$-orbit $Y$ there exists $f_Y\in k(X)^{(B)}$ with $n_D=\nu_D(f_Y)$ for any
  $B$-stable divisor $D$ containing $Y$.
\end{thm}

\begin{proof}
Assume that $\delta$ is Cartier. We may assume that $X$ is simple
with closed orbit $Y$. By Sumihiro's Theorem $X$ is quasi-projective
and we may assume that $\delta$ is globally generated (any Cartier
divisor on a quasi-projective variety can be written as the
difference of two globally generated divisors). Since $\delta$ is
globally generated, there exists a section $\sigma\in
H^0(X,\cO_X(\delta))^{(B)}$ with $\sigma\vert_Y\neq0$ and $\delta$
is principal on $X_\sigma$.

Conversely, we may again assume that $X$ is simple with closed orbit $Y$. Replacing $\delta$ by
$\delta-\div(f_Y)$ we may assume that no component of $\delta$ contains $Y$. We may therefore replace $\delta$ by $D$
an irreducible $B$-stable divisor not containing $Y$.
Denote by $i:X^{\rm reg}\to X$ the inclusion of the regular locus.
The locus where the sheaf $i_*\cO_{X^{\rm reg}}(D\cap X^{\rm reg})$
is not locally free is a $G$-stable closed subvariety (up to taking
a fnite cover of $G$ so that the sheaf is $G$-linearised) and
contained in $D$. It is therefore empty (since $Y$ is the only
closed $G$-orbit).
\end{proof}

\begin{defn}Let $X$ be a spherical variety

(\i) Denote by $\cC^\vee(X)$ the union of the cones $\cox$ for $Y$ a $G$-orbit.

(\i\i) Define the set $PL(X)$ of piecewise linear functions as the
subgroup of functions $l$ on $\cC^\vee(X)$ such that
\begin{itemize}
  \item for any $G$-orbit $Y$, the restriction $l_Y$ to $\cox$ is the
    restriction of an element of $\X(G/H)$;
\item for any $G$-orbit $Z$ with $Z\subset\overline{Y}$, we have
  $l_Z\vert_{\cox}=l_Y$.
  \end{itemize}

(\i\i\i) Denote by $L(X)$ the group of linear functions on $\F(X)$.
\end{defn}

\begin{remark}
  The function $l$ only depends on its values on the maximal
  cones therefore on the cones of the closed orbits in
  $X$.
  Note however that if the maximal cones do not have the
  maximal dimension then the function $l_Y$ does not determine an element in $\X(G/H)$.
\end{remark}

\begin{thm}
There is an exact sequence
$$\cC^\vee(X)^\perp\cap\X\to\Z(\mathcal{D}\setminus\mathcal{D}(X))\to\pic(X)\to
PL(X)/L(X)\to 0.$$
\end{thm}

\begin{proof}
By Proposition \ref{prop-cartier}, we have a surjective map
$\pic(X)\to PL(X)/L(X)$ defined by sending $\delta$ to the
collection $(f_Y)_Y$. The kernel is given  by
divisors $\delta=\sum_Dn_D[D]$ such that $n_D=0$ for all $D\in
  \mathcal{D}_Y(X)$ for some orbit $Y$ thus coming from
  $\Z(\mathcal{D}\setminus\mathcal{D}(X))$.
If a Cartier divisor is principal its image lies in $L(X)$. If $f$
is such that $\div(f)$ is in the kernel, then for all $Y$ and $D$ a
$B$-stable divisor with $D\supset Y$, we have
$\scal{\rho(\nu_D),f}=0$ thus $f\in\cC^\vee(X)^\perp$.
\end{proof}

\begin{cor}
$\pic(X)$ is of finite rank.
\end{cor}

\begin{cor}
\label{pic-sans-tors} $\pic(X)$ is torsion free if there exists a
complete $G$-orbit.
\end{cor}

\begin{proof}
The quotient $PL(X)/L(X)$ is free. For $Y$ a complete $G$-orbit we
have $\dim\cox=\dim\X^\vee$ by Theorem \ref{theo-rk} thus
$\cC^\vee(X)^\perp=0$.
\end{proof}

\begin{remark}
If $X$ has no complete $G$-orbit then the Picard group may have torsion as shows the example
$X=G$ with $G$ semisimple not simply connected (this is a spherical variety for the action of $G\times G$).
\end{remark}

\subsubsection{Ample and globally generated line bundles}

If $X$ is complete, the function $l$ completely determines for each
closed $G$-orbit $Y$ an element $\chi_Y\in\X(G/H)$ such that
$l_Y=\chi_Y$ so that the value $l_Y(\rho(\nu_D))=\nu_D(\chi_Y)$ is
well defined even if $\rho(\nu_D)$ does not lie in $\cox$. This will
simplify the statement of the next result for which we will assume
that $X$ is complete. For a characterisation of globally generated
and ample line bundles in the general setting, see \cite[Theorem
17.3]{timashev}.

\begin{defn}
A function $l\in PL(X)$ is called convex if $l_Y\leq (l_Z)\vert_\cox$ for any orbit $Y$ and any closed
orbit $Z$. If the inequality is strict, then $l$ is called strictly convex.
\end{defn}

\begin{thm}
\label{theo-pic}
Assume $X$ complete and consider a Cartier $B$-stable divisor
$$\delta=\sum_{D\in \mathcal{D}\setminus\mathcal{D}(X)}n_DD+ \sum_{D\in \mathcal{D}(X)\cup\mathcal{V}(X)}l_Y(\rho(\nu_D)).$$

(\i) The divisor $\delta$ is globally generated if and only if $l$ is convex and the inequality $l_Y(\rho(\nu_D))\leq n_D$ holds
for all closed $G$-orbit $Y$ and all $D\in\mathcal{D}\setminus\mathcal{D}(X)$.

(\i\i) The divisor $\delta$ is ample if and only if $l$ is strictly convex and the inequality $l_Y(\rho(\nu_D))< n_D$ holds
for all closed $G$-orbit $Y$ and all $D\in\mathcal{D}\setminus\mathcal{D}(X)$.
\end{thm}

\begin{proof}
{\it (\i)} We may assume that $\delta$ is $G$-linearised. It is then
globally generated if and only for any closed $G$-orbit $Y$ here
exists $\s\in H^0(X,\delta)^{(B)}$ such that $\s\vert_Y\neq0$. This
in turn is equivalent to the fact that we can write
$\delta=A_Y+\div(f_Y)$ with $A_Y$ effective not containing $Y$ and
$f_Y$ a $B$-semiinvariant function. If $\chi_Y$ is the weight of
$f_Y$, then this is equivalent to the inequalities $l(\rho(\nu_D))\geq
\chi_Y(\rho(\nu_D))$ for all $D\in \mathcal{D}(X)\cup\mathcal{V}(X)$
with equality for $D\supset Y$ and $\chi_Y(\rho(\nu_D))\leq n_D$ for
all $D\in\mathcal{D}\setminus\mathcal{D}(X)$.

{\it (\i\i)} By the previous argument, a large multiple of $\delta$
separates closed orbits if and only if  $l_Y\neq l_Z$ for $Y$ and
$Z$ two distinct closed orbits. The proof therefore reduces to the
case of a simple spherical variety $X$ with closed orbit $Y$ in
which case we may choose $l_Y=0$. If $\delta$ is ample, then
$n\delta-\sum_{D\in \mathcal{D}\setminus\mathcal{D}_Y(X)}D$ is
globally generated for large $n$ thus $n_D>0$ by {\it (\i)}.
Conversely, assume $n_D>0$ for all $D\in
\mathcal{D}\setminus\mathcal{D}_Y(X)$ and let $\eta$ be the
canonical section of the effective divisor $\delta$. We have
$X_\eta=X_{Y,B}$ which is affine and one easily checks that any
$f\in k[X_\eta]$ is of the form $\sigma/\eta^n$ for some $n$ and
some $\sigma\in H^0(X,n\delta)$. Since $(X_{g\cdot\eta})_{g\in G}$
is a covering, $\delta$ is ample.
\end{proof}

\begin{cor}
For $X$ spherical and complete, any ample divisor is globally generated.
\end{cor}

\begin{cor}
\label{nef=gg}
For $X$ spherical and complete, nef and globally generated line bundles agree.
\end{cor}

\begin{proof}
Recall from Corollary \ref{cor-polehedral}, that rational and
numerical equivalence agree for Cartier divisors. Theorem
\ref{theo-pic} implies that the cone of globally generated line
bundles is the closure of the ample cone. The latter is the nef cone
(see \cite[Theorem 1.4.23]{Lazarsfeld}).
\end{proof}

\subsubsection{Quasi-projective and affine spherical varieties}

The description of ample line bundles on spherical varieties leads
to the following characterisation of quasi-pro\-jective spherical
varieties.

\begin{cor}
A spherical variety $X$ is quasi-projective if and only if there
exists a strictly convex $\Q$-valued function on $\cC^\vee(X)$ which
is linear on each cone $\cox$.
\end{cor}

In \cite{knop}, Knop gives a characterisation of affine spherical
varieties.

\begin{thm}
A spherical variety $X$ is affine if and only if $X$ is simple and
there exists $\chi\in\X$ such that $\chi\vert_{\mathcal{V}}\leq0$,
$\chi\vert_{\cC^\vee(X)}=0$ and
$\chi\vert_{\rho(\mathcal{D}\setminus\mathcal{D}(X))}>0$.
\end{thm}

\subsubsection{A criterion for $\Q$-factoriality}

Comparing the group of Weil divisors with the Picard group gives the
following result.

\begin{thm} \label{thm-Q-fact}
Let $X$ be a spherical variety.

(\i) The variety $X$ is locally factorial if and only if for any
$G$-orbit $Y$, the elements $\rho(\nu_D)$ lying in $\cox$ for $D$
any $B$-stable divisor form a $\Z$-linearly independent subset of
$\X^\vee$.

(\i\i) The variety $X$ is locally $\Q$-factorial if and only if for
any $G$-orbit $Y$, the elements $\rho(\nu_D)$ lying in $\cox$ for
$D$ any $B$-stable divisor form a linearly independent subset of
$\X^\vee_\Q$.
\end{thm}

\begin{proof}
Since this is local we may assume that $X$ is simple with closed
orbit $Y$. By Theorem \ref{theo-chow}, we have an exact sequence
$\X\to\Z(\mathcal{V}(X)\cup\mathcal{D})\to A_{\dim X-1}(X)\to0$.
Since there is a unique cone in the fan, we have $PL(X)=L(X)$ thus
for the Picard group we have an exact sequence
$\cox^\perp\cap\X\to\Z(\mathcal{D}\setminus\mathcal{D}_Y(X))\to\pic(X)\to
0$, by Theorem \ref{theo-pic}. The groups $\pic(X)$ and $A_{\dim
X-1}(X)$ coincide if and only if for any
$D\in\mathcal{V}(X)\cup\mathcal{D}_Y(X)$ there exists $\chi_D\in\X$
with $\chi_D(\rho(\nu_{D'}))=\delta_{D,D'}$ which is the required
condition.
\end{proof}

\subsubsection{Comments}
The description of the Picard group of spherical varieties follows
\cite{brion-nombre-car}. See also \cite{timashev}. The above
characterisation of $\Q$-factorial spherical varieties can be found
in \cite{brion-lectures}. There is no general description of big
divisors. For $X$ complete, the cone of effective divisors is
convex polyhedral generated by the classes of $B$-stable divisors
(Corollary \ref{cor-polehedral}), but it is hard to give in general
explicit generators of this cone. Partial answers for effective
divisors on wonderful varieties or for big and effective divisors on
symmetric varieties are given in \cite{brion-magni} and
\cite{ruzzi-big}.

%%%%%%%%%%%%%%%%%%%%%%%%%%%%%%%%%%%%%%%%%%%%%%%%%%%%%%%%%%%%%%%%%%%%%%%%%%%%%%

\subsection{Toroidal varieties}
\label{section-toro}

In this Section we assume $\char(k)=0$. This section is one of the
two sections where we consider a special class of spherical
varieties. There are several reasons for this. First, any spherical
variety has a toroidal open dense subset whose complement is in
codimension at least $2$. The computation of a canonical divisor for
toroidal varieties therefore leads to a computation of a canonical
divisor of any spherical variety. A second reason relies on the fact
that the $\Q$-factoriality criterion (Theorem \ref{thm-Q-fact})
turns out to be a smoothness criterion for toroidal varieties. This
then leads to an explicit equivariant resolution of singularities
for any spherical variety. Finally, using toroidal varieties, we
prove that the set of invariant valuations $\mathcal{V}$ is a
polyhedral convex cone.

\begin{defn}
A spherical variety $X$ is toroidal if $\mathcal{D}(X)$ is empty.
\end{defn}

\subsubsection{Local structure}

Let $X$ be a spherical variety and let $\Delta_X=\cup_{D\in
  \mathcal{D}}\overline{D}$. Denote by $P_X$ the stabiliser of
$\Delta_X$, it is easy to check that $P_X$ is also the stabiliser of
$BH/H$ and in particular it is a $G$-birational invariant of $X$.

\begin{prop}
\label{prop-str-tor}
 Let $X$ be a spherical variety. The following
are equivalent.

(\i) The variety $X$ is toroidal.

(\i\i) There exists a Levi subgroup $L$ of $P_X$ depending only on
$G/H$ and a closed subvariety $Z$ of $X\setminus \Delta_X$ stable
under $L$ such that the map
$$(P_X)_u\times Z\to X\setminus\Delta_X$$
is an isomorphism. The group $[L,L]$ acts trivially on $Z$ which is
a toric variety for a quotient of $L/[L,L]$. Furthermore any
$G$-orbit meets $Z$ along a unique $L$-orbit.
\end{prop}

\begin{proof}
Assume that $X$ is toroidal, since $\Delta_X$ is Cartier and
globally generated, we may apply Theorem \ref{theo-loc-str} to
$\eta$ the canonical section of $\co_X(\Delta_X)$. For some closed
$L$-spherical variety $Z\subset X\setminus\Delta_X$, we have
$X_\eta=X\setminus\Delta_X\simeq (P_X)_u\times Z$. Furthermore
$B\cap L$ has finitely many orbits in $Z$ and one easily checks that
$(BH/H)\cap Z=(G/H)\cap Z$ thus $L=(B\cap L)H$ which implies that
$[L,L]$ acts trivially on $Z$ (see \cite[Section 3.4, Lemme]{BLV}).
We get that $Z$ is a toric variety for a quotient of $L/[L,L]$.

Conversely, any $G$-orbit of $X$ meets $Z$ and thus is not contained
in $\Delta_X$. It is therefore not contained in any $B$-stable non
$G$-stable divisor.
\end{proof}

\begin{remark}
\label{rema-orb}
  The previous result implies that there is a one to one
  correspondence between the $G$-orbits in $X$ and the $L$-orbits in
  $Z$: a $L$-orbit $Z'$ in $Z$ is mapped to $GZ'$ while a $G$-orbit
  $Y$ in $X$ is mapped to $Y\cap Z$.
\end{remark}

\begin{cor}
Assume that $X$ is toroidal. Then $X$ is smooth if and only if for
any $G$-orbit $Y$, the cone $\cox$ is generated by a basis of
$\X^\vee$.
\end{cor}

\begin{proof}
The previous result implies that $X$ has the singularities of a toric
variety with the same cones as the cones
of $X$. The result
follows by standard result on toric varieties (see
\cite[Section 2.1, Proposition 1]{fulton-toric} or \cite[Theorem 1.10]{Oda}).
\end{proof}

\subsubsection{Valuation cone} Using toroidal embedding we may prove that
the cone $\mathcal{V}$ of invariant valuations is a polyhedral
convex cone.

\begin{prop}
\label{prop-tor}
There exists a complete toroidal embedding of $G/H$.
\end{prop}

\begin{proof}
Choose $f$ in
$k[G]^{(B\times H)}$ such that $f$ is $H$-invariant and vanishes on
the inverse image in $G$ of any $D\in\mathcal{D}$. Let $W$ be the
$G$-module spanned by $f$ and let $X'$ be the closure of the image
of the induced morphism $G/H\to\p(W^\vee)$. There is no $G$-orbit of
$X'$ contained in the divisor $\textrm{div}_0(f)$. If $X''$ is a
complete embedding of $G/H$, then the normalisation $X$ of the
closure of $G/H$ diagonally embedded in $X'\times X''$ gives the
desired embedding.
\end{proof}

\begin{cor}
\label{cor-cone} The set $\mathcal{V}$ is a polyhedral convex cone.
\end{cor}

\begin{proof}
First one has to prove that $\mathcal{V}$ is convex. For this result
we refer to \cite[Proposition 2.1]{pauer} or \cite[Lemma 5.1]{knop}.
Then pick $X$ a complete toroidal embedding. The union of the
finitely many polyhedral convex cones $\cox$ for $Y$ a closed
$G$-orbit is $\mathcal{V}$.
\end{proof}

\subsubsection{Resolution of singularities} As an application we get equivariant resolutions.

\begin{cor}
Let $X$ be a spherical variety, there exists a $G$-birational
morphism $\Xt\to X$ with $\Xt$ toroidal.
\end{cor}

\begin{proof}
Replace any colored cone $(\cC,\cF)$ by
$(\cC\cap\mathcal{V},\emptyset)$.
\end{proof}

\begin{cor}
Any spherical variety has a $G$-equivariant toroidal resolution.
\end{cor}

\begin{proof}
Take a $G$-birational toroidal embedding dominating $X$. Subdivising its
fan as for toric varieties (see \cite[Section 2.6]{fulton-toric} or
\cite[Section 1.5]{Oda}) gives a resolution.
\end{proof}

\subsubsection{Canonical divisor}

Define the canonical sheaf $\omega_X$ on a normal variety $X$ by
extending the canonical sheaf of the smooth locus:
$\omega_X=i_*\omega_{X^{\rm reg}}$ where $i:X^{\rm reg}\to X$ is the
embedding of the smooth locus of $X$. A Weil divisor $K_X$ is called
canonical if $\cO_X(K_X)=\omega_X$. Denote by $\partial X$ the union
of the $G$-stable divisors.

\begin{thm}
\label{theo-can}
  There exists a canonical divisor $K_X$ of $X$ such that
$$-K_X=\sum_{X_i\in \mathcal{V}(X)}X_i+\sum_{D\in
    D(G/H)}a_D\overline{D}=\partial X+H$$
with $a_D$ positive integers and $H$ a globally generated divisor.
\end{thm}

\begin{proof}
Replacing $X$ by the union of $G$-orbits of codimension at most one,
we may assume that $X$ is toroidal and smooth. We want to study the
action map ${\rm Act}:\g\otimes\co_X\to T_X$ obtained from
differentiating the map $G\times X\to X$ given by the action of $G$
on $X$. First restrict the tangent bundle to
$X\setminus\Delta_X\simeq(P_X)_u\times Z$, we get
$T_{X\setminus\Delta_X}\simeq((\mathfrak{p}_X)_u\otimes\cO_{(P_X)_u})\times
T_Z$. Since $Z$ is a toric variety for a quotient $S$ of $L/[L,L]$,
if $\mathfrak{s}$ is the Lie algebra of $S$, we have that the image
of the action map for this toric variety is
$\mathfrak{s}\otimes\co_Z\simeq T_Z(-{\rm log}\partial Z)$ where
$\partial Z$ is the union of the $S$-stable divisors in $Z$ and
$T_Z(-{\rm log}\partial Z)$ is the logarithmic tangent bundle
obtained as subsheaf of $T_X$ of derivations of $\cO_X$ preserving
the ideal of $\partial Z$ (see \cite[Proposition 3.1]{Oda}). Because
of the correspondence between orbits in $X$ and in $Z$ (see Remark
\ref{rema-orb}) we get that the action map over $X\setminus\Delta_X$
has image $T_{X\setminus\Delta_X}(-{\rm log}\partial X)$ which is
free. Let $h=\dim H$ and $\G(h,\mathfrak{g})$ be the Gra\ss mann
variety of vector subspaces of dimension $h$ in $\mathfrak{g}$. The
kernel of the action map is locally free of rank
  $h$ and defines a morphism
  $X\setminus\Delta_X\to\G(h,\mathfrak{g})$ which extends the natural
  morphism $x\mapsto\mathfrak{h}_x$ for $x\in G/H$ and for
  $\mathfrak{h}_x$ the Lie algebra of the stabiliser of $x$. By
  $G$-equivariance we get a morphism $\varphi:X\to\G(h,\mathfrak{g})$
  and an exact sequence $0\to\varphi^*S\to\mathfrak{g}\otimes\cO_X\to
  T_X(-\log\partial X)\to0$, where $S$ is the tautological subbundle
  on $\G(h,\g)$. Taking the maximal exterior power we
  get an isomorphism $\varphi^*\cO_{\G(h,\mathfrak{g})}(1)\simeq
  \omega_X^{-1}(-\partial X)$ and the result follows after picking a
  well chosen $B$-stable divisor representing
  $\cO_{\G(h,\mathfrak{g})}(1)$.
\end{proof}

\subsubsection{Rigidity}

Toroidal varieties satisfy rigidity properties. It is well known
(see for example \cite[Theorem
  4.6]{kodaira-deformation}) that a smooth complex variety $X$ has no
deformation if $H^1(X,T_X)=0$. Write $S_X=T_X(-\log\partial X)$,
then the pair $(X,\partial X)$ has no deformation if $H^1(X,S_X)=0$.

\begin{thm}
\label{theo-vanis}
Let $X$ be a smooth spherical variety.

(\i) If $X$ is toroidal, then $H^i(X,S^*S_X)=0$ for $i>0$.

(\i\i) If furthermore $G/H$ is proper over an affine, then
  $H^i(X,\mathcal{L}\otimes S^*S_X)=0$ for $i>0$ and any globally
  generated line bundles $\mathcal{L}$.
\end{thm}

\begin{proof}
We only give a brief sketch of proof and refer to \cite[Theorem
  4.1]{knop-annals} for the first part and to \cite[Proposition
  3.2]{bien-brion} for the second part.

{\it (\i)} Let $L_X=\spec S^*S_X$ and consider the Stein factorisation
  $L_X\to M_X\to \mathfrak{g}^\vee$ of the moment map (see Section
\ref{section-convex} for more on the moment map). Then the hard part is to
prove that $M_X$ has rational singularities and that the fiber of $L_X\to M_X$ is unirational. The main result of \cite{kollar} finishes the proof.

{\it (\i\i)} The assumption implies the existence of an effective
  $\Q$-divisor $D$ containing $\partial X$ such that
  $\cO_{\p(S_X^\vee)}(1)\otimes p^*\cO_X(E)$ is big and nef with $p$ the
  morphism $p:\p(S_X^\vee)\to X$. Kawamata-Viehweg Theorem
  (see \cite{kawamata-matsuda-matsuki} or \cite{viehweg}) gives the
  result.
\end{proof}

For Fano varieties, this result implies a rigidity result (see
\cite[Proposition 4.2]{bien-brion}).

\begin{cor} Let $X$ be a smooth toroidal spherical Fano variety, then
  we have $H^i(X,T_X)=0$ for all $i>0$.
\end{cor}

\begin{proof}
Write $(X_i)_{i\in[1,n]}$ for the irreducible components of
$\partial X$. If $\mathcal{N}_{X_i}$ is the normal bundle of $X_i$,
then $\mathcal{N}_{X_i}=\omega_X^{-1}\otimes\omega_{X_i}$. The exact
sequence $0\to S_X\to T_X\to\oplus_{i=1}^n\mathcal{N}_{X_i}\to 0$,
the above vanishing $H^i(X,S_X)=0$ and, since $\omega_X^{-1}$ is
ample, the Kodaira vanishing Theorem
$H^1(X_j,\omega_X^{-1}\otimes\omega_{X_i})=0$ give the vanishing.
\end{proof}

\subsubsection{Comments}
Proposition \ref{prop-str-tor} is taken from \cite{brion-pauer}. The
description of the set $\mathcal{V}$ as a cone was first proved in
\cite{brion-pauer} and extended to any $G$-variety in any
characteristic in \cite{knop-bewert}. There is an explicit formula
for the coefficients of the canonical divisor given in
\cite{brion-aus}. The first part of Theorem \ref{theo-vanis} is
proved in \cite{knop-annals} and the second part as well as the
rigidity of Fano regular varieties  can be found in
\cite{bien-brion}. These vanishing results were partly extended to
quasi-regular varieties in \cite{boris}\footnote{A quasi-regular
variety is a smooth spherical variety such that the components of
the boundary are smooth and
$\rho(\mathcal{D}(X))\subset\mathcal{V}$.}.

%%%%%%%%%%%%%%%%%%%%%%%%%%%%%%%%%%%%%%%%%%%%%%%%%%%%%%%%%%%%%%%%%%%%%%%%%%%%%%%

\subsection{Sober spherical varieties}
\label{section-sober}

In this section we assume $\char k=0$. This is the second section
where we consider special classes of spherical varieties. The first
reason for this is that sober varieties and especially symmetric
varieties and completions of algebraic groups played a leading role
in the development of the theory of spherical varieties and still
represent an important source of inspiration. The second reason is
that wonderful varieties \emph{i.e.} the smooth projective toroidal
simple sober varieties play a crucial role in the classification of
homogeneous spherical varieties.

\subsubsection{Structure of the valuation cone and little Weyl group}
\label{subsect-swg}

The set of $G$-invariant valuations $\mathcal{V}$ of $G/H$ is a
convex cone (see Corollary \ref{cor-cone}). We give more precise
results on its structure. First we characterise the vertex of this
cone. Let $\aut_G(G/H)=N_G(H)/H$ be the $G$-automorphism group of
$G/H$. We refer to \cite[Section 4 and Theorem 6.1]{knop} for a
proof.

\begin{prop}
\label{prop-cone-aut}
 The group $\aut_G(G/H)$ is a diagonalisable group. In particular, the
  connected component of the identity $\aut_G(G/H)^0$ is a central torus.

Furthermore we have $\dim \aut_G(G/H)=\dim(\mathcal{V}\cap
-\mathcal{V})$.
\end{prop}

In particular the cone $\mathcal{V}$ is strictly convex if and only if $N_G(H)/H$ is finite.
If this is the case, the subgroup $H$ is called
\emph{sober}.

\begin{cor}
\label{cor-simple-compl} $H$ is sober if and only if $G/H$ has a
simple toroidal completion.
\end{cor}

\begin{proof}
Take $({\mathcal{V}},\emptyset)$ as colored cone.
\end{proof}

\begin{defn}
An embedding of $G/H$ with $H$ a sober subgroup is called sober.
\end{defn}

To describe the structure of the cone of valuations
$\mathcal{V}$ we may assume by Proposition \ref{prop-cone-aut} that
$H$ is sober.
In that case Brion \cite{brion-demazure} proved the following
result.

\begin{thm}
\label{thm-cone} The cone $\mathcal{V}$ is the fundamental domain of
the Weyl group $W_X$ of a root system in $\X^\vee$. In particular it
is simplicial.
\end{thm}

\begin{proof}
The main idea of the proof is to use the fact that fundamental
domains of the Weyl group of a root system are characterised by the
values of the angles between two edges of the cone (which can only
be $\pi/2$. $\pi/3$, $\pi/4$ or $\pi/6$). This is proved by
restriction to the rank $2$ case and Brion proves that for this it
is enough to consider spherical varieties of rank at most $4$. The
proof then follows by case by case inspection.

The root system can be explicitly described as follows. There is a
unique simple toroidal completion $X$ of $G/H$ (Corollary
\ref{cor-simple-compl}) and the Structure Theorem for $X$ (see
Proposition \ref{prop-str-tor}) gives an affine toric variety $Z$
for some torus $C$ which is a quotient of $T$ and with
$\cX(C)=\X(G/H)$. The generators of the monoid of all weights of
elements in $k[Z]^{(C)}$ yields the root system in $\cX(C)=\X(G/H)$.
\end{proof}

\begin{defn}The group $W_{G/H}$ is called the little Weyl group of $G/H$.
\end{defn}

As the above proof does not give much geometric information or
reason why this result should be true, Knop started a project to
give a more conceptual definition of the little Weyl group $W_{G/H}$
as well as a more conceptual proof of the fact that $\mathcal{V}$ is
fundamental domain for $W_{G/H}$. He achieved this and even
generalised most of the statements to any $G$-variety in a series of
papers \cite{knop-bewert}, \cite{knop-asymptotic}, \cite{knop-jams},
\cite{knop-weylgruppe}, \cite{knop-annals}. See also the text
\cite{knop-habil}. We will give a very brief review of his results
here.

Assume $k=\mathbb{C}$. First using the moment map, Knop defines a
little Weyl group in a geometric way and gives many properties of
its action. Here we consider the moment map in the complex setting
as follows. For $X$ a smooth $G$-variety let $T^\vee_X$ be the total
space of its cotangent bundle, then we have a map
$$\mu:T^\vee_X\to\g^\vee$$
defined by $(x,\xi)\mapsto(\eta\mapsto \xi(d_{e,x}\sigma(\eta,0)))$
where $\sigma:G\times X\to X$ is the action. We recover the real
moment map (see Subsection \ref{section-convex}) from this map since
the choice of an invariant hermitian form gives a section of the
cotangent bundle. Now we can consider the following fiber product
$$\xymatrix{{T}^\vee_X\times_{\mathfrak{t}/W}\mathfrak{t}\ar[rr]\ar[d] &&\mathfrak{t}\ar[d]\\
T^\vee_X\ar[r] & \g^\vee\simeq\g\ar[r]
&\g/\!\!/G\simeq\mathfrak{t}/W.\\}$$ The variety
${T}^\vee_X\times_{\mathfrak{t}/W}\mathfrak{t}$ is in general not
irreducible so pick an irreducible component $C$ dominating
$T^\vee_X$. The Weyl group acts on the irreducible components of
${T}^\vee_X\times_{\mathfrak{t}/W}\mathfrak{t}$.

\begin{thm}
The little Weyl group $W_X$ is the quotient
$N_W(C)/C_W(C)$. In other words, $W_X$ is the Galois group of the
covering $C\to T_X^\vee$.
\end{thm}

\subsubsection{Automorphism free spherical varieties}

Recall that sober spherical varieties have a unique simple complete toroidal embedding $\bar X$. The following two problems are natural.
\begin{itemize}
\item[(1)] Give an explicit description of $\bar X$.
\item[(2)] Is $\bar X$ smooth?
\end{itemize}

\begin{defn}
A smooth simple complete toroidal embedding is called a wonderful
embedding.
\end{defn}

There are complete answers to these problems for automorphism free
spherical varieties \emph{i.e.} for $H=N_G(H)$. For $H$ satisfying
this condition, define the Demazure embedding $X_D$ as the closure
of the $G$-orbit of the Lie algebra $\h$ of $H$ in the Gra\ss mann
variety $\G(h,\g)$ with $h=\dim\h$. This is a non necessarily normal
completion of $G/H$. Brion \cite{brion-demazure} proves that the
inclusion of $G/H$ in $X_D$ extends to a morphism $\bar X\to X_D$
and that this morphism is the normalisation map. He conjectured that
this morphism is an isomorphism and that $X_D$ is smooth. This is
now verified.

\begin{thm}
If $H=N_G(H)$, then the variety $\bar X$ is smooth and $\bar X=X_D$.
\end{thm}

The first assertion was proved by Knop in \cite{knop-jams} and the
second by Losev in \cite{losev}. Note that for $H$ a spherical
subgroup, its normaliser $N_G(H)$ is such that $G/N_G(H)$ is
automorphism free (see \cite[Lemma 30.2]{timashev}).

\begin{remark}
\label{rem-knop} Knop proves even more: if for all $n\in N_G(H)$
acting trivially on $\mathcal{D}$ we have $n\in H$, then $\bar X$ is
smooth.
\end{remark}

\subsubsection{Symmetric varieties}
\label{subsection-symmetric}
 Let $G$ be a reductive group and let $\theta$ be a group involution of
 $G$. A closed subgroup $H$ of $G$ such that $G^\theta\subset H\subset N_G(G^\theta)$ is called a symmetric
 subgroup. Vust \cite{vust1} proved that any symmetric subgroup is spherical.

\begin{defn}
\label{defi-symm} An embedding of $G/H$ with $H$ a symmetric
subgroup is called a symmetric variety.
\end{defn}

The classification of all embeddings is carried out in \cite{vust},
we review this briefly here. First, let $S$ be a torus in $G$
maximal for the property: $\theta(s)=s^{-1}$ for all $s\in S$. Let
$T$ be a maximal torus of $G$ containing $S$. This torus is stable
under $\theta$ which induces an involution on the root system $R$ of
$(G,T)$. The restricted root system $R_{G,\theta}$ is defined as
$$R_{G,\theta}=\{\a-\theta(\a)\ /\ \a\in R\}.$$
This is a root system. Furthermore, there exists a Borel subgroup
$B$ containing $T$ such that for $\a\in R^+$ we have $\theta(\a)=\a$
or $\theta(\a)\in R^-$. In that case $BH/H$ is dense in $G/H$ and
$S_{G,\theta}=\{\a-\theta(\a)\ /\ \a\ \textrm{simple root}\}$ is a
basis of $R_{G,\theta}$. Denote by $S_{G,\theta}^\vee$ the set of
simple coroots and by $C^\vee$ the dominant chamber of the dual root
system $R^\vee_{G,\theta}$.

\begin{thm}
Let $H$ with $G^\theta\subset H\subset N_G(G^\theta)$.

Then $\X=\cX(S/S\cap H)$, $\mathcal{V}=-C^\vee$ and
$\rho(\mathcal{D})=S_{G,\theta}^\vee$. Furthermore the fibers of the
map $\rho:\mathcal{D}\to S_{G,\theta}^\vee$ have at most two
elements.
\end{thm}

Note that $\X$, $\mathcal{V}$ and $\rho(\mathcal{D})$ are fixed by
$G$ and $\theta$ but the set of colors $\mathcal{D}$ depends on the
subgroup $H$.

\begin{remark}
The classification of group involutions by Cartan \cite{cartan}
gives a classification of all symmetric subgroups. This yields a
complete classification of symmetric varieties via colored cones and
restricted root systems. The classification of spherical varieties
proposed by Luna in \cite{luna} is directly motivated by this
classification.
\end{remark}

\begin{example}
Any reductive group $G$ can be realised as a symmetric $G\times
G$-variety: $G\simeq G\times G/{\rm diag}(G)$ where ${\rm diag}(G)$
is the diagonal embedding of $G$. It is the fixed point subgroup of
the involution $\theta(x,y)=(y,x)$. In that case the restricted root
system is simply the root system of $G$.
\end{example}

When the group $G$ is adjoint, it is easy to check that the
symmetric varieties are automorphism free so that there exists a
unique smooth simple toroidal completion. These completions are
called \emph{complete symmetric varieties}. The existence of this
compactification was first proved in \cite{deCP}. The complete
symmetric varieties have very nice enumerative properties. The most
classical example being the variety of complete conics in the
cohomology ring of which the Steiner problem can be solved: there
are $3264$ conics tangent to $5$ given conics in general position.

\subsubsection{Wonderful varieties and classification}
\label{subsection-wonderful}

Wonderful varieties are generalisations of complete symmetric
varieties.

\begin{defn}
\label{def-wonder}
A $G$-variety $X$ is wonderful if the following properties hold.
\begin{itemize}
\item $X$ is smooth and projective.
\item $G$ has an open orbit $\Omega$ in $X$.
\item The boundary $\partial X=X\setminus\Omega$ is a divisor with simple normal crossing. The intersection of all irreducible components $(D_i)_{i\in[1,n]}$ of $\partial X$ is non empty.
\item The $G$-orbit closures are the partial intersections of the divisors $(D_i)_{i\in[1,n]}$.
\end{itemize}
\end{defn}
As the definition suggests, wonderful varieties are well understood
from the point of view of log-geometry. In particular wonderful
varieties are the complete log-homogeneous varieties with a unique
closed orbit. One can also characterise them as the log-homogeneous
and log-Fano varieties. We refer to
\cite{brion-log},\cite{brion-log2} for more on log-homogeneous
varieties. Luna \cite{luna2} proved that wonderful varieties are
spherical.

\begin{thm}
The wonderful varieties are the wonderful embeddings.
\end{thm}

The importance of spherical varieties relies on the fact that their
classification is equivalent to the classification of spherical
varieties. Indeed, Luna in \cite{luna} defines, for $H$ a spherical
subgroup its \emph{spherical closure} $\bar H$ as the subgroup of
elements in $N_G(H)$ acting trivially on $\mathcal{D}$. He proves
that the classification of the spherical subgroups can be deduced
from the classifications of their spherical closure. The advantage
is that $\bar H$ is sober (since $(N_G(\bar H)/\bar H)^\circ$ leaves
$\mathcal{D}$ pointwise fixed) and by a result of Knop (see Remark
\ref{rem-knop}) the simple toroidal completion of $G/\bar H$ is
smooth therefore it is a wonderful variety.

The main problem is now to determine, for $X$ a wonderful variety
which spherical system \emph{i.e.} triples
$(\X^\vee(X),\mathcal{V}(X),\mathcal{D}(X))$ actually occur and to
describe all the varieties associated to this triple. Losev
\cite{losev} proved that for each triple, there exists at most one
wonderful variety $X$ having the corresponding triple of invariants.
In \cite{luna}, Luna replaces the triple by the triple
$(\X^\vee(X),\Sigma(X),\mathcal{D}(X))$ where $\Sigma_X$ is the set
of simple roots generating the cone $\mathcal{V}(X)$ (see Theorem
\ref{thm-cone}). In the case of symmetric varieties we recover the
restricted root system. Luna gives an axiomatic definition of
spherical systems and the problem is to construct a wonderful
variety associated to each abstract spherical system. Recently two
solutions, the first via combinatorics in \cite{bravi-pezza},
\cite{bravi-pezzb} and \cite{bravi-pezzc} and the second via
geometry in \cite{cupit} were proposed.

\subsubsection{Comments}

Proposition \ref{prop-cone-aut} was first proved in \cite[Section
5.2]{brion-pauer}, see \cite[Section 4 and Theorem 6.1]{knop} for a
characteristic free statement. The geometric description of the
little Weyl group $W_X$ was first given in \cite{knop-weylgruppe}.
The fact that the cone $\mathcal{V}$ is a fundamental domain for
$W_X$ is proved in \cite{knop-asymptotic}. Note that the set of
invariant valuations forms a cone $\mathcal{V}$ for any $G$-variety
as was proved by Knop in \cite{knop-bewert} over any algebraically
closed fields. Knop conjectures that $\mathcal{V}$ should be the
fundamental domain of a Weyl group also in positive characteristic.
This question remains open.

%%%%%%%%%%%%%%%%%%%%%%%%%%%%%%%%%%%%%%%%%%%%%%%%%%%%%%%%%%%%%%%%%%%%%%%%%

\section{Further geometric properties}

%%%%%%%%%%%%%%%%%%%%%%%%%%%%%%%%%%%%%%%%%%%%%%%%%%%%%%%%%%%%%%%%%%%%%%%%%

\subsection{Mori theory}
\label{section-mori}

There is a general framework for studying the birational geometry of
complex projective variety called Mori theory. The minimal model
program (MMP)
has been proved in dimension at most three with $\char k=0$ but is
still not available in full generality in higher dimension. Large
parts of it have been now realised for $\char k=0$ but very few
results are available in general for $\char k>0$.
In this section we will explain that the MMP works perfectly for
spherical varieties in all characteristics.

\subsubsection{Mori dream spaces} A conceptual reason for the MMP to
run perfectly can be stated as follows.

\begin{thm}
\label{thm-mds1} A projective spherical $\Q$-factorial variety is a
Mori dream space.
\end{thm}

Mori dream spaces were introduced by Hu and Keel \cite{HK} who
proved that for these varieties, the MMP works perfectly in all
characteristics (see in particular \cite[Proposition 1.11]{HK}).
Recall a definition of Mori dream spaces\footnote{Note
  that this is not the usual definition of Mori dream spaces, see
  \cite[Definition 1.10 and Theorem 2.9]{HK}.}.

\begin{defn}
(\i) Let $(D_i)_{i\in[1,n]}$ be a family of Cartier divisors. Then
we
  write
  $$R(X,(D_i)_{i\in[1,n]})=\bigoplus_{(r_i)_{i\in[1,n]}\in\Z^n}H^0(X,\cO_X(\raisebox{0pt}[\height][0pt]{$\sum$}_ir_iD_i)).$$
This is a $\Z^n$-graded ring.

(\i\i) A Mori dream space is a normal projective $\Q$-factorial
variety $X$ with $N^1(X)$ finitely generated and having a basis
$([D_i])_{i\in[1,n]}$ of $N^1(X)$ with $D_i$ Cartier for all $i$
and such that the ring $R(X,(D_i)_{i\in[1,n]})$ is finitely
generated.
\end{defn}

\begin{proofofthm}
By Corollaries \ref{cor-polehedral} and \ref{pic-sans-tors}
$\pic(X)=N^1(X)$ is torsion free and finitely generated.
Furthermore, the nef cone is a rational polyhedral convex cone (see
Theorem \ref{theo-pic}). Let $(D_i)_{i\in[1,n]}$ be a set of
generators of this group. The following lemma concludes the proof.
\end{proofofthm}

\begin{lemma}
\label{elm-fg}
  Let $(D_i)_{i\in[1,n]}$ be a collection of Cartier divisors on a
  spherical variety $X$, then $R(X,(D_i)_{i\in[1,n]})$ is finitely
  generated.
\end{lemma}

\begin{proof}
Consider the vector bundle $V$ over $X$ associated to the locally
free sheaf $\oplus_i\cO_X(-D_i)$. The variety $V$ is affine over $X$
and is spherical for the group $G\times\G_m^n$. Since
$R(X,(D_i)_{i\in[1,n]})=H^0(V,\cO_V)$ we are left to prove that for
$X$ spherical $H^0(X,\cO_X)$ is finitely generated.

We only have to prove that $H^0(X,\cO_X)^U$ is finitely generated
where $U$ is a maximal unipotent subgroup: in characteristic 0, if
$H^0(X,\cO_X)^U$ is finitely generated, then it generates a $G$-stable
subalgebra $R$. If $A$ if a complementary $G$-module then $A^U=0$
thus $A=0$. In positive characteristics, see \cite[Theorem
9]{grosshans}. But $H^0(X,\cO_X)^U$ is generated by
$H^0(X,\cO_X)^{(B)}$ and since there is a dense $B$-orbit each
eigenspace has dimension at most 1. Let $\Gamma$ be the monoid of
the weights of $H^0(X,\cO_X)^{(B)}$, we have to prove that $\Gamma$
is finitely generated.

Let $\X(\Gamma)$ the subgroup of $\cX(T)$ generated by $\Gamma$. Let
$\Omega$ be the dense $B$ orbit in $X$. The complement of $\Omega$
is a finite union of $B$-stable divisors
$(D)_{D\in\mathcal{D}\cup\mathcal{V}(X)}$. We have
$\Gamma=\{\lambda\in\X(\Gamma)\ /\ \nu_D(f_\lambda)\geq0\
\textrm{for} \ f_\lambda\in H^0(X,\cO_X)^{(B)}_\lambda\
\textrm{and}\ D\in\mathcal{D}\cup\mathcal{V}(X)\}$. Therefore
$\Gamma$ is a polyhedral convex cone and by Gordan's Lemma it is
finitely generated.
\end{proof}

\subsubsection{Mori program}
The following results, known as contraction Theorem, existence of
flips and termination of flips are standard consequences of the
above result. We give here sketches of proof in the spherical
variety setting.

Recall from Corollary \ref{cor-polehedral}, that the cone of
effective curves $NE(X)$ on a projective spherical variety $X$ is
convex polyhedral. For $f:X\to Y$ with $X$ projective, we denote by
$NE(X/Y)$ the cone of effective curves contracted by $f$.

\begin{thm}
\label{thm-mds} Let $X$ be a projective spherical variety.

(\i) For each face $F$ of $NE(X)$, there is a unique spherical
variety $X_F$ and a unique $G$-morphism $cont_F:X\to X_F$ with
connected fibers such that $F=NE(X/X_F)$.

(\i\i) The face $F$ generates the kernel of $(cont_F)_*:N_1(X)_\Q\to
N_1(X_F)_\Q$ and $N^1(X_F)_\Q$ can be identified to the orthogonal
of $F$ in $N^1(X)_\Q$.

(\i\i\i) Any morphism $\varphi:X\to Y$ with $Y$ projective such that
$F\subset NE(X/Y)$ factors through $cont_F$.
\end{thm}

\begin{proof}
We only construct the contraction and prove the factorisation
property. Pick $D$ a Cartier divisor such that $D$ is positive on
$NE(X)\setminus F$ and vanishes on $F$. Then
$A=\oplus_nH^0(X,\cO_X(nD))$ is a finitely generated
normal $G$-algebra. Define $X_F=\proj(A)$.
This yields the contraction.

If $\varphi:X\to Y$ satisfies $F\subset NE(X/Y)$, then consider
$F'=NE(X/Y)$ and we have that $F'/F$ is a face of $NE(X_F)$. As
above there is a morphism $Cont_{F'/F}:X_F\to X'=(X_F)_{F'/F}$. Then
by unicity we have that $Cont_{F'}=Cont_{F'/F}\circ Cont_F$ and this
morphism is also the Stein factorisation of $\varphi$ concluding the
proof.
\end{proof}

\begin{thm}
Let $X$ be a $\Q$-factorial projective spherical variety and let $R$
be an extremal ray of $NE(X)$ such that $Cont_R:X\to X_R$ is an
isomorphism in codimension one.

Then there exists a unique $\Q$-factorial projective spherical
variety $X^+$ and a unique birational morphism $\varphi^+:X^+\to
X_R$ called the flip of $Cont_R$ such that
\begin{itemize}
\item $\varphi^+=Cont_{R^+}$ is the contraction of an extremal ray $R^+$ in
  $NE(X^+)$.
\item $\varphi^+$ is an isomorphism in codimension one.
\item The spaces $N^1(X)$ and $N^1(X^+)$ are identified via
  $\varphi^+\circ Cont_R^{-1}$ and the half-lines $R$ and $R^+$ are
  opposite in $N_1(X)$.
\end{itemize}
\end{thm}

\begin{proof}
We will only construct the morphism $\varphi^+:X^+\to X_R$. Let $C$
be a curve in $X$ such that $[C]$ spans $R$ and let $D$ be a Cartier
divisor with $D\cdot C<0$. Then $A=\oplus_n(Cont_R)_*\cO_X(nD)$ is
finitely generated (use the Local Structure Theorem to pass from the
relative setting to the global setting and use Lemma \ref{elm-fg})
and up to taking some power of $D$ it is a normal $G$-algebra.
Define $X^+=\proj(A)$. This yields the flip.
\end{proof}

Let $D$ be a divisor and let $Cont_{R^+}:X^+\to X_R$ be the flip of
$Cont_R:X\to X_R$. We say that the flip is $D$-directed if $D\cdot
R<0$ and $D\cdot R^+>0$.

\begin{thm}
There is no infinite sequence of $D$-directed flips.
\end{thm}

\begin{proof}
First remark that $X^+$ and $X$ are birational and isomorphic in
codimension one. In particular they share the same sets of
$B$-stable and $G$-stable divisors. Having these two sets fixed,
there are only finitely many possible fans for $X^+$
 and by Theorem \ref{thm-class-fan} this implies that there are only finitely many possible
embeddings of $G/H$ obtained by flips from $X$. In particular, there
exists a $\Q$-factorial embedding $\Xt$ of $G/H$ which dominates all
the flips obtained from $X$. Let $X^+$ be one of these flips, then
we have morphisms $f:\Xt\to X$ and $f^+:\Xt\to X^+$ and we claim
that
$$D_\Xt=f^*D+\sum_ia_i E_i={f^+}^*D^++\sum_ia_i^+ E_i$$
with $a_i^+\geq a_i$ and inequality for some index $i$ (here $D^+$
is the strict transform of $D$). This will conclude the proof. To
prove this assertion remark that ${f^+}^*D^+-f^*D$ is nef relatively
to the composition $\Xt\to X\to X_R$ and since any relatively nef
divisor is globally generated (this is a relative version of
Corollary \ref{nef=gg}) the result follows.
\end{proof}

\subsubsection{Comments}
The results of this section are taken from \cite{brion-mori} and
\cite{brion-knop}. The proof of Lemma \ref{elm-fg} is taken from
\cite{knop-hilbert}. Most of these results can be stated in a more
general context: all the results remain true for the relative
minimal model program over a $G$-spherical variety $S$. Furthermore,
these statements are true for varieties of complexity $1$. See
\cite{brion-knop} for these generalisations. In \cite{brion-mori}, a
more precise study of the cone of curves as well as the description
of part of the above results in terms of colored fans (see Section
\ref{section-class}) has been achieved.

%%%%%%%%%%%%%%%%%%%%%%%%%%%%%%%%%%%%%%%%%%%%%%%%%%%%%%%%%%%%%%%%%%%%%%%%%

\subsection{Frobenius splitting}
\label{sec-split}

In this section we prove that starting with a spherical variety in
characteristic $0$, its reductions modulo a prime $p$ are Frobenius
split for almost all $p$. We deduce different proofs of
Corollary \ref{cor-sing-rat} and
Theorem \ref{thm-vanish}.

\subsubsection{Existence of a splitting}
For a variety $X$ defined over an algebraic closed field $k$ of
positive characteristic $p$, we say that $X$ is Frobenius split if
the \emph{absolute} Frobenius morphism $F:X\to X$ has a section on
the level of structural sheaves \emph{i.e.} the morphism $\cO_X\to
F_*\cO_X$ has a section $\varphi:F_*\cO_X\to \cO_X$. A subvariety
$Y\subset X$ is compatibly split for $\varphi$ if its sheaf of
ideals $\cI_Y$ satisfies $\varphi(F_*\cI_Y)\subset\cI_Y$. We will
use the book \cite{BK} for further reference on Frobenius
splittings.

Let $X$ be spherical and defined over an algebraically closed field
of characteristic~$0$ and write $X_p$ for its reduction to an
algebraically closed field of characteristic $p$.

\begin{thm}
The variety $X_p$ is Frobenius split with splitting $\varphi$ for
all but finitely many $p$ and $\varphi$ compatibly splits all the
closed $G$-stable subvarieties.
\end{thm}

\begin{proof}
Since any spherical variety has a completion which itself has a
resolution by a complete toroidal variety, we may assume $X$ smooth
complete and toroidal (using \cite[Lemma 1.1.7 and 1.1.8]{BK}). In
particular $X_p$ is smooth for all except finitely many primes.
Since any closed $G$-variety is the intersection of $G$-stable
divisors, it is enough to prove that $X$ is split by a $(p-1)$-power
of a section of $\omega_X$ compatibly splitting all $G$-stable
divisors (see \cite[Proposition 1.2.1 and Theorem 1.4.10]{BK}).

Theorem \ref{theo-can} implies that for $Y$ a closed orbit in $X$,
we have $\omega_Y^{-1}=H\vert_Y$ with $H$ globally generated. Since
$Y_p$ is projective homogeneous  there exists $\sigma\in
H^0(Y_p,\omega_{Y_p}^{-1})$ such that $\sigma^{p-1}$ splits $Y$. But
in characteristic zero, $H^0(Y,\omega_Y^{-1})$ is irreducible and
$H$ globally generated therefore the map $H^0(X,H)\to
H^0(Y,\omega_Y^{-1})$ is surjective. This is still true for $X_p$,
$Y_p$ for $p$ large enough thus there is a lift $\tau$ of $\sigma$.
Multiply $\tau$ by the canonical section of $\partial X$ and take a
$(p-1)$-power to get the desired splitting.
\end{proof}

\begin{cor}
Let $X$ be a spherical variety and $Y$ a closed $G$-stable
subvariety, then for all but finitely many $p$ the variety $Y_p$ has
rational singularities.
\end{cor}

\begin{proof}
  It is enough to prove that $X_p$ has rational singularities for
  large $p$ and that $Y_p$ is normal for large $p$ since the $Y_p$ is
  again a spherical variety. We may assume $X$
  affine and consider a toroidal resolution $\pi:\Xt\to X$. Pick a
  relative ample $B$-stable divisor $D$. By Theorem \ref{theo-can}, we
  have $(1-p)K_{\Xt_p}\geq D$ for $p$ large. In particular the splitting
  which compatibly splits the closed $G$-stable subvarieties is a
  $D$-splitting by \cite[Remark 1.4.2]{BK}. Furthermore since
  $\pi$ is $G$-equivariant, the splitting vanishes on the exceptional
  locus. By \cite[Theorem 1.2.8 and 1.3.14]{BK},
we get $H^0(X_p,R^i\pi_*\cO_{\Xt_p})=H^i(\Xt_p,\cO_{\Xt_p})=0$ for
  $i>0$ thus $R^i\pi_*\cO_{\Xt_p}=0$ and $R^i\pi_*\omega_{\Xt_p}=0$ for
  $i>0$. This proves that $X$ has rational singularities. We
  also get, for $\Yt$ a closed $G$-stable subvariety of $\Xt$, surjections
  $H^0(\Xt,\pi^*\mathcal{L})\to H^0(\Yt,\pi^*\mathcal{L}\vert_\Yt)$ for
  $\mathcal{L}$ globally generated on $X$. By \cite[Lemma 3.3.3]{BK}
  this proves $\pi_*\cO_\Yt=\cO_Y$ for $Y=\pi(\Yt)$ proving the
  normality of the closed $G$-stable subvarieties in $X$.
\end{proof}

\begin{cor}
Assume that $X$ is proper over an affine, then for all but finitely
many $p$ and for $\mathcal{L}$ nef on $X$, we have
$H^i(X_p,\mathcal{L}_p)=0$ for $i>0$ and the map
$H^0(X_p,\mathcal{L}_p)\to H^0(Y_p,\mathcal{L}_p)$ is surjective for
all $Y$ closed and $G$-stable in $X$.
\end{cor}

\begin{proof}
Pick a toroidal resolution $\pi:\Xt\to X$. By the previous result we
have $\pi_*\cO_{\Xt_p}=\cO_{X_p}$ and $R^i\pi_*\cO_{\Xt_p}=0$ for
$i>0$ so the results follow from the toroidal case explained in the
proof of the previous Corollary.
\end{proof}

\subsubsection{Comments}
Frobenius splitting techniques were first introduced by Mehta and
Ramanathan \cite{mehta}, \cite{mehta-ramanathan} to study Schubert
varieties. For some special classes of spherical varieties, more
precise results have been obtained. Any toric variety is Frobenius
split in any characteristic (see for example \cite{payne}). The
spherical $G\times G$-embeddings of a reductive group $G$ are also
Frobenius split in any characteristic and it can be proved that they
admit a rational resolution so that they are Cohen-Macaulay (see
\cite{strickland} or \cite[Chapter 6]{BK}). In \cite{tange}, R.
Tange explains how to extend these splitting results to some other
cases using generalised parabolic  induction techniques. More
generally, in any characteristic different from $2$, any embedding
of a symmetric variety (see Definition \ref{defi-symm}) is Frobenius
split compatibly with the closed $G$-stable subvarieties (this was
proved in \cite{deconcini-springer} for the Wonderful
compactification, it is easily extended to all embeddings, see
\cite{perrin-wahl}). This was used in \cite{perrin-wahl} to study
the Gau\ss\  map in positive characteristic for cominuscule
homogeneous spaces.

Some much more precise splitting results have been obtained by X. He
and J.F. Thomsen in \cite{HT1} and \cite{HT2}. In particular they
applied these results to describe the singularities of $B\times B$-orbit
closures in spherical $G\times G$-embeddings of a reductive group
$G$ (see also Subsection \ref{plus-loin-orbites}).

\subsection{Convex geometry}
\label{sec-conv}

In this section we assume $\char k=0$ and $X$ to be a projective. We
associate to $X$ several convex polytopes and see that many
geometric properties of $X$ can be described in terms of these
polytopes.

We prove that the moment map gives rise to another characterisation
of spherical varieties and that the classification by colored cones
can partially be translated into convex geometry of the image of the
moment map. Using more general convex polytopes we construct flat
deformation to toric varieties, give a smoothness criterion and
describe a subring of the cohomology ring of spherical varieties.

\subsubsection{Moment map}
\label{section-convex} Assume $k=\mathbb{C}$. For $V$ a $G$-module
and $X$ a projective $G$-variety equivariantly embedded in $\p(V)$,
there exists a moment map $\mu:X\to \ka^\vee$ where $\ka$ is the Lie
algebra of a maximal compact subgroup $K$ of $G$.
The map can be defined as follows. Let $\scal{\ ,\ }$ be an hermitian
$K$-invariant scalar product on
$V$, then $\mu(x)(\kappa)=i\frac{\scal{x,\kappa\cdot
    x}}{\scal{x,x}}$. If $T$ is a maximal torus of $G$ such that
$T_K=T\cap K$ is a maximal torus in $K$, then the dual
$\mathfrak{t}_K^\vee$ of the Lie algebra of $T_K$ is the subspace of
$T_K$-invariants in $\ka^\vee$. This subspace contains the dominant
chamber $\mathfrak{t}_{K,+}^\vee$. The following is a classical
result.

\begin{thm}
\label{theo-image-moment} $\mu(X)\cap\mathfrak{t}_{K,+}^\vee$ is
convex.
\end{thm}

In this context, Brion gave a new proof of this result by
reinterpreting the convex polytope
$\mu(X)\cap\mathfrak{t}_{K,+}^\vee$. We shall see that the polytope
contains many information on the geometry of the variety $X$.
Let $\mathcal{L}$ be an ample $G$-linearised line bundle on $X$ and let
$R(X,\mathcal{L})=\oplus_nH^0(X,\mathcal{L}^{\otimes n})$.

\begin{defn}
\label{def-moment-poly}
Set $\Delta(X,\mathcal{L})=\{\lambda\ /\ \textrm{$\exists\ n$\ :\ $n\lambda\in \cX(T)$ and $H^0(X,\mathcal{L}^{\otimes n})^{(B)}_{n\lambda}\neq0$}\}$.
\end{defn}

Because the algebra of $U$-invariants $R(X,\mathcal{L})^U$ is
integral of finite type, the above set $\Delta(X,\mathcal{L})$ is a
convex polytope (see for example \cite[Proposition
  2.1]{brion-moment}). Furthermore one easily checks that it spans an
affine space with vector space direction $\X_\Q$ in $\cX(T)_\Q$ (see
\cite[Proposition 1.2.3]{brion-lectures}).
Assume that $\mathcal{L}=\cO_{\p(V)}(1)$.

\begin{prop}
\label{prop=p} We have
$\Delta(X,\mathcal{L})=\mu(X)\cap\mathfrak{t}_{K,+}^\vee\cap
\cX(T)_\Q$.
\end{prop}

\begin{proof}
The main idea is that for $\lambda$ dominant, results of Mumford
(see \cite[Appendix]{ness}) imply that the weight $\lambda/n$ lies
in $\mu(X)$ if and only if $X\times G/P_\lambda$ has a stable point
with respect to the polarisation $\mathcal{L}^{\otimes
  n}\otimes\cO_{G/P_\lambda}(1)$ (here $P_\lambda$ is the parabolic
subgroup stabilising $\lambda$). This in turn is equivalent to the
existence of a non-trivial invariant and therefore a non trivial map
$V_{m\lambda}\to H^0(X,\mathcal{L}^{\otimes mn})$ proving the
result.
\end{proof}

We now characterise spherical varieties thanks to the moment map.

\begin{prop}
\label{prop-moment}
Let $X$ be a normal $G$-variety, then $X$ is spherical if and only if
for each $x\in X$, the fibre $\mu^{-1}(\mu(x))$ is an orbit under the
stabiliser $K_{\mu(x)}$ of $\mu(x)$.
\end{prop}

\begin{proof}
We keep the notation of the previous proof. A result of Kirwan
\cite{kirwan} implies: for $\lambda$ dominant, with $\lambda/n\in
\mu(X)$, the quotient $\mu^{-1}(-\lambda/n)/K_{-\lambda/n}$ is
homeomorphic to the GIT quotient $(X\times G/P_\lambda)/\!/G$ (with
polarisation $\mathcal{L}^{\otimes n}\otimes\cO_{G/P_\lambda}(1)$ as
in the proof of Proposition \ref{prop=p}). In particular this
quotient is connected. If $X$ is spherical, then the GIT quotient is
finite and thus reduced to a point. Conversely, if the quotient of
the fiber is a point, so is the GIT quotient thus the algebra
$R(X\times G/P_\lambda,\mathcal{L}^{\otimes
n}\otimes\cO_{G/P_\lambda}(1))^G$ is of dimension 1. Its degree 1
piece $\Hom(V_\lambda,H^0(X,\mathcal{L}^{\otimes n}))^G$ is thus of
dimension at most 1 proving that $X$ is spherical by Theorem
\ref{theo-char}.
\end{proof}

We describe the connection between the colored fan of $X$ and the
convex polytope $\Delta(X,\mathcal{L})$. Let $\sigma\in
H^0(X,\mathcal{L})^{(B)}$, we can write
$$\div(\s)=\sum_{D\in\mathcal{V}(X)\cup\mathcal{D}}n_DD,$$
with $n_D\geq0$. Let $\lambda_\sigma$ be the $B$-weight of $\s$.

\begin{prop}
$\Delta(X,\mathcal{L})=\{\lambda_\s+\xi\ /\
\textrm{$\scal{\rho(\nu_D),\xi}+n_D\geq0$ for $D\in
  \mathcal{V}(X)\cup\mathcal{D})$}\}$
\end{prop}

\begin{proof}
This comes from the fact that a section of $H^0(X,\mathcal{L}^{\otimes
  n})$ is of the form $\sigma^nf$ for some $f\in k(X)$ and defines an
effective divisor.
\end{proof}

Some of the fan geometry can be translated into convex geometry of the
polytope  $\Delta(X,\mathcal{L})$. We only state the following
result.

\begin{thm}
\label{thm-poly=cone} Let $Y$ be a $G$-orbit $Y$.

(\i) $\Delta(\overline{Y},\mathcal{L})$ is a face of
$\Delta(X,\mathcal{L})$.

(\i\i) $\cox$ is the dual cone:
  $\cox=\{\nu\in\X^\vee_\Q\ /\ \scal{\nu,\xi}\geq0\,\ \forall\xi\in\Delta(\overline{Y},\mathcal{L})\}$.

(\i\i\i) $\mathcal{D}_{Y}(X)= \{D\in\mathcal{D}\ /\ \scal{\rho(\nu_D),
    \xi-\lambda_\sigma}+n_D=0,\  \forall\xi\in
  \Delta(\overline{Y},\mathcal{L})\}$.

(\i v) The map $Y\mapsto \Delta(\overline{Y},\mathcal{L})$ is a
  bijection from the $G$-orbits onto the set of faces of
  $\Delta(X,\mathcal{L})$ whose interior meets $\mathcal{V}$.
\end{thm}

\subsubsection{Deformation to toric varieties}
\label{sec-def}

There is a general framework steaming from convex geometry which
leads to the construction of flat deformation of projective
varieties to toric varieties. This was first initiated by Okounkov
\cite{okounkov1} and \cite{okounkov2} and developed in several
directions (see \cite{lazarsfeld-mustata}, \cite{anderson}). For
spherical varieties, this was first developed by Kaveh
\cite{kaveh-sp-spherique} and Alexeev and Brion \cite{brion-alexeev}
inspired in many aspects by the work of Caldero \cite{caldero} on
toric degeneration of projective rational homogeneous spaces.

\begin{defn}
Let $\Gamma$ be a semigroup in $\N\times\Z^d$, we define the cone $C(\Gamma)$
as the cone in $\R\times\R^d$ generated by $\Gamma$ and the Okounkov body as
  $\Delta(\Gamma)=C(\Gamma)\cap ({1}\times R^d)$.
\end{defn}

\noindent Choose a total ordering which respects addition on $\Z^d$,
a $\Z^d$-valuation on a field $K$ is a map
$\nu:K\setminus\{0\}\to\Z^d$ such that
$\nu(f+g)\geq\min(\nu(f),\nu(g))$ and $\nu(fg)=\nu(f)+\nu(g)$.

\begin{example}
\label{ex-flag-val} Let $X$ be a projective variety of dimension $n$
and $(Y_i)_{i\in[1,n]}$ a complete flag of subvarieties with $Y_i$
normal of dimension $i$, then we may define a valuation
$\nu_{Y_\bullet}:k(X)\to\Z^n$ by the order of functions on the
subvarieties $Y_i$: $\nu_{Y_\bullet}(f)=({\rm
ord}_{Y_1}(f),\cdots,{\rm
  ord}_{Y_n}(f))$.
\end{example}

\begin{example}
\label{ex-val} Starting with a line bundle $\mathcal{L}$ on $X$ and
$\nu$ a $\Z^d$-valuation on $k(X)$, we may define the corresponding
Newton-Okounkov body.
Define the semigroup
$$\Gamma=\Gamma_\nu(X,\mathcal{L})=\{(k,\nu(f))\in\N\times\Z^d\ /\ \textrm{for
} f\in
  H^0(X,\mathcal{L}^{\otimes k})\}.$$
The cone $C(\Gamma)$ and the
  Okounkov body $\Delta(\Gamma)$ will be in this case denoted by
  $C_\nu(X,\mathcal{L})$ and $\Delta_\nu(X,\mathcal{L})$.
\end{example}

\begin{example}
\label{ex-poids}
The moment body $\Delta(X,\mathcal{L})$ is an Okounkov body for the
semigroup obtained using the weights of $B$-semiinvariant functions by
setting
$$\Gamma(X,\mathcal{L}) =
\{(k,\lambda)\in\N\times\cX(T)\ /\  H^0(X,
\mathcal{L}^{\otimes k})^{(B)}_\lambda\neq0\}.$$
\end{example}

Let $\Gamma$ be a semigroup in $\N\times\Z^d$ and let
$R=\oplus_nR_n$ be a graded ring with a filtration
$(\cF_{\gamma})_{\gamma\in\Gamma}$ indexed by $\Gamma$ satisfying
the following conditions:
\begin{itemize}
\item for $\gamma\leq\gamma'$, we have $\cF_\gamma\subset\cF_{\gamma'}$
\item $\cF_\gamma\cdot\cF_{\gamma'}\subset\cF_{\gamma+\gamma'}$.
\end{itemize}
Denote by $\textrm{Gr}_{\cF}(R)$ the graded algebra associated to
the filtration. Denote by $\cF_{\leq\gamma}$ (resp. $\cF_{<\gamma}$)
the union of the $\cF_{\gamma'}$ with $\gamma'\leq\gamma$ (resp.
$\gamma'<\gamma$). We say that the filtration has leaves of
dimension one if $\dim\cF_{\leq\gamma}/\cF_{<\gamma}\leq1$.

\begin{example}
Starting with a valuation coming from a flag of subvarieties as in
Example \ref{ex-flag-val}, then the filtration obtained from the
lexicographical order on $\Z^n$ by $\cF_\gamma=\{f\ /\
\nu(f)\geq\gamma\}$ has leaves of dimension one (see \cite[Lemma
1.3]{lazarsfeld-mustata} ).
\end{example}

\begin{thm}
Assume that $\Gamma$ is finitely generated.

(\i) Then there exists a
finitely generated subalgebra $\mathcal{R}\subset R[t]$ such that
\begin{itemize}
\item $\mathcal{R}$ is flat over $k[t]$;
\item $\mathcal{R}[t^{-1}]\simeq R[t,t^{-1}]$;
\item $\mathcal{R}/t\mathcal{R}\simeq\textrm{Gr}_\cF(R)$.
\end{itemize}

(\i\i) If furthermore the filtration has leaves of dimension one and if
$\Gamma$ contains all the rational points $(\N\times\Z^d)\cap
C(\Gamma)$ of the cone $C(\Gamma)$,
then the limit $\proj(\textrm{Gr}_\cF(R))$ is normal and is the toric
variety associated to the convex polytope $\Delta(\Gamma)$.
\end{thm}

\begin{proof}
{\it (\i)} This proof is now classical, we follow \cite{caldero} and
\cite{brion-alexeev}. The ring $\textrm{Gr}_\cF(R)$ is finitely
generated. Pick generators $(\overline{f}_i)$ of degree $\gamma_i$
lift them in $(f_i)$ in generators of $R$. Consider
the morphism of graded rings $S=k[x_i]\to R$, $x_i\mapsto
f_i$.
Pick generators $\overline{g}_k$ of degree $\nu_k$ of the kernel. We
have $\overline{g}_k(f_i)\in R_{<\nu_k}$ and we may find $g_k\in
\overline{g}_k+S_{<\nu_k}$ with $g_k(f_i)=0$. Since the associated
graded map is an isomorphism, so is the map $S/(g_k)\to R$. Pick a
linear map $\Z\times\Z^d\to\Z$ such that $\pi$ is positive on $\gamma_i-\nu_k$
and the generators of $\Z\times \Z^d$ (for the existence see \cite[Lemma
  5.2]{anderson}). Define $R_{\leq j}$ as the span of the monomials in
$f_i$ whose degree has value at most $j$ under the map $\pi$. Then
the ring $\mathcal{R}=\oplus_nR_{\leq j}t^j$ satisfies the
conditions.

{\it (\i\i)} The filtration being with leaves of dimension one, the
limit algebra $\textrm{Gr}_\cF(R)$ is the algebra $k[\Gamma]$ of the
semigroup $\Gamma$. Furthermore, the condition on the semigroup is
classically equivalent to the normality (see \cite[Section 2.1
Proposition 2]{fulton-toric} or \cite[Proposition 1.2]{Oda}).
\end{proof}

We may now apply this result to spherical varieties. This is based on
two main facts. First the multiplicity-free property of spherical
varieties, second combinatorial properties of dual canonical
bases.

Let $X$ be a projective spherical variety and let $\mathcal{L}$ be
an ample $G$-linearised line bundle. Let
$R=R(X,\mathcal{L})=\oplus_nH^0(X,\mathcal{L}^{\otimes n})$. This
ring is multiplicity-free \emph{i.e.}
\begin{equation}
\label{eq-R}
R_n=\bigoplus_{(n,\lambda)\in\Gamma(X,\mathcal{L})}
V(\lambda),
\end{equation}
with
$V(\lambda)$ the representation of highest weight $\lambda$ and
$\Gamma(X,\mathcal{L})$, the semigroup defined in Example
\ref{ex-poids}.

The dual canonical basis $(v_{\lambda,\phi})$
is a basis of $V(\lambda)$ for all $\lambda$. Furthermore there exists
a so called \emph{string} parametrisation by elements in $\Z^d$ (with
$d=\dim U$) satisfying the following conditions (see
\cite{littelman-canonique}). Let
$\Gamma_{can}=\{(\lambda,\phi)\in\cX(T) \times
\Z^d\ /\ \exists\ v_{\lambda,\phi}\neq0\}$.

\begin{thm}
$\Gamma_{can}$ is the intersection of a polyhedral cone
  $\cC_{can}$ with $\cX(T)\times\Z^d$.
\end{thm}

Let $\Delta_{can}$ be the Okounkov body of $\Gamma_{can}$. The
multiplicative properties of the dual canonical basis (proved in
\cite{caldero}) imply that the set
$$\Gamma_{str}=\{(n,\lambda,\phi)\in\N\times\cX(T)\times\Z^d\ /\ (n,\lambda)\in\Gamma(X,\mathcal{L}),\ (\lambda,\phi)\in\Gamma_{can}\}$$
is a semigroup called the string semigroup of $R$. We have an
isomorphism of graded module
\begin{equation}
\label{eq-R-mod}
R\simeq k[\Gamma_{str}]
\end{equation}
 and the
previous results easily imply.

\begin{cor}
Let $X$ be a projective spherical variety with $\mathcal{L}$ ample,
then there exists a flat deformation to a toric variety $X_0=\proj
k[\Gamma_{str}]$
\end{cor}

\begin{cor}
\label{cor-sing-rat-def} Projective\footnote{This statement is also
true for affine spherical varieties. Indeed, in the affine case, the
same arguments lead to the existence of a deformation to an affine
normal toric variety.} spherical varieties have rational
singularities.
\end{cor}

\begin{proof}
Since toric varieties have rational singularities (see \cite[Section
3.5 last Proposition]{fulton-toric}, \cite[Corollary 3.9]{Oda}),
this follows from the stability of rational singularities under
deformations (see \cite{elkik}).
\end{proof}

\subsubsection{Okounkov bodies and degree of line bundles}
\label{subsection-okounkov}

After Okounkov \cite{okounkov1} and La\-zarsfeld and Musta{\c t}\u a
\cite{lazarsfeld-mustata}, Kaveh and Khovanskii study, in a series of
papers \cite{KK-annals},\cite{KK2}, the relationship between
the geometry of varieties (in particular the growth of sections of
line bundles) with the convex geometry of some semigroups in
$\N\times\Z^d$. In particular they prove the following approximation
result (see \cite[Theorem 1.13]{KK-annals}).

\begin{thm}
Let $\Gamma$ be a semigroup in $\N\times\Z^d$, let $\Delta$ be the
corresponding Okounkov body and let $H_\Gamma(k)$ be
the number of elements in $\Gamma\cap(\{k\}\times\Z^d)$.

(\i) Then
$H_\Gamma(k)\sim\textrm{vol}(\Delta)k^q$
where $q=\dim\Delta$ and $\textrm{vol}(\Delta)$ is the volume of the convex
body $\Delta$\footnote{The volume is normalised so that the volume of
  the quotient $\R\times\R^d/\Z\times\Z^d$ is 1.}.

(\i\i) Let $F$ be a polynomial function of degree $n$ on $\R^d$ and let $f$
be its homogeneous part of degree $n$, then
$$\lim_{k\to\infty}\frac{\sum_{(k,x)\in\Gamma}F(x)}{k^{n+q}}=\int_{\Delta}f(x)dx$$
where $dx$ is the Lebesgue measure on $\Delta$.
\end{thm}

Let $X$ be spherical of dimension $n$. We can apply this result in a
straightforward way using the description of the ring
$R(X,\mathcal{L})$ for $\mathcal{L}$ an ample line bundle. Indeed,
Equation (\ref{eq-R}) and Equation (\ref{eq-R-mod}) give
$$\dim H^0(X,\mathcal{L}^{\otimes
  k})=\sum_{(k,\lambda)\in\Gamma(X,\mathcal{L})}\dim V(\lambda)
\sim\textrm{vol}(\Delta_{str})k^{\dim\Delta_{str}}.$$ The function
$\dim V(\lambda)$ is a polynomial function on the characters. It is
given by the formula
$F(\lambda)=\prod_{\a>0}\frac{\scal{\a,\rho+\lambda}}{\scal{\a,\rho}}$.
Set
$f(\lambda)=\prod_{\a>0}\frac{\scal{\a,\lambda}}{\scal{\a,\rho}}$.

\begin{cor}
\label{cor-deg}
We have
$\deg\mathcal{L}=n!\textrm{vol}(\Delta_{str})=n!
\int_{\Delta(X,\mathcal{L})}f(x)dx.$
\end{cor}

\subsubsection{Smoothness criterion}
\label{subsection-smooth}

The same type of arguments were used by Brion \cite{brion-lissite}
to give a smoothness criterion for spherical varieties. Since this
is a local problem, we may assume that $X$ is simple with closed
orbit $Y$. The main idea is to compute the multiplicity of the ideal
of $Y$\footnote{The multiplicity of an ideal $I$ in a local ring $A$
is
  the highest term in $k$ of the function $\textrm{lg}(A/I^n)$.} using an
integral on a simplex. Let $\cox$ be the cone associated to $X$ and
let $\cC_Y(X)$ be its dual. Let $(f_i)$ be generators of $\cC_Y(X)$
and define
$$\cC(Y)=\cC_Y(X)\setminus\bigcap_i(f_i+\cC_Y(X)).$$
Recall the definition of the function $f$ above.

\begin{thm}
The multiplicity of the ideal of $Y$ is given by
$$c!\int_{\cC(Y)}f(x)dx$$
where $c$ is the codimension of $Y$ in $X$.
\end{thm}

\begin{cor}
The variety $X$ is smooth if and only if the cone $\cox$ is generated
by a basis of $\X^\vee$ and $\cC(Y)$ is a simplex of volume
$1/\codim(Y)!$ for the form $f(x)dx$.
\end{cor}

\subsubsection{Intersection theory}
\label{subsection-inter}

Kaveh explained in \cite{kaveh-intersection} how to describe the
subring of the cohomology ring $H^*(X,\R)$ generated by the Picard
group $\pic(X)$ thanks to the formula for the degree of any line
bundle. The main observation is the following result.

\begin{thm}
Let $A=\sum_{i=0}^nA_i$ be a commutative, finite dimensional graded
$k$-algebra, generated in degree 1, satisfying Poincar{\'e} duality and
such that $A_0=A_n=k$.

Let $(H_i)_{i\in[1,r]}$ be a basis of $A_1$ and define the polynomial
$P$ by $P(x_1,\cdots, x_r)=\left(\sum_ix_iH_i\right)^n\in A_n$. Then
the algebra $A$ is isomorphic to $k[Y_1,\cdots,Y_r]/I$ where $I$ is
the ideal
$$I=\left\{f(Y_1,\cdots,Y_r)\ /\ f\left(\frac{\partial}{\partial
  x_1},\cdots,\frac{\partial}{\partial x_r}\right)\cdot P=0\right\}.$$
\end{thm}

Let $X$ be a spherical projective variety of dimension $n$ and fix
$(H_i)_{i\in[1,r]}$ a basis of $\pic(X)$. For $\mathcal{L}$ an ample
line bundle, write $\mathcal{L}=\sum_ix_iH_i$ and define
$P(x_1,\cdots,x_r)=\deg(X,\mathcal{L})=n!\textrm{vol}(\Delta_{str})$.

\begin{cor}
Assume that the subalgebra of $H^*(X,\R)$ generated by $\pic(R)$
satisfies Poincar{\'e} duality, then this algebra is isomorphic to
$\R[Y_1,\cdots,Y_r]/I$ where $I$ is the ideal
$$I=\left\{f(Y_1,\cdots,Y_r)\ /\ f\left(\frac{\partial}{\partial
  x_1},\cdots,\frac{\partial}{\partial x_r}\right)\cdot P=0\right\}.$$
\end{cor}

This result in particular generalises Borel's presentation of the
cohomology ring of complete flag varieties \cite{borel} and the description by
\cite{khovanskii-pukhlikov} of the cohomology ring of toric
varieties.

\subsubsection{Comments} Theorem \ref{theo-image-moment} was initiated by Guillemin
and Steinberg \cite{gui-steinI}, \cite{gui-stein2} and proved in
full generality in \cite{kirwan}. Proposition \ref{prop-moment} was
proved in a more general setting by Atiyah \cite{atiyah}. We follow
the proof of Brion \cite{brion-moment}. The results on the moment
map are mainly taken from \cite{brion-moment}. The proof of Theorem
\ref{thm-poly=cone} can be found in \cite[Proposition
5.3.2]{brion-lectures}.

Our presentation of toric degenerations follows
\cite{brion-alexeev}, \cite{anderson} and \cite{KK2}. In a recent
preprint \cite{kaveh-crystal}, Kaveh recovers the parametrisation of
the dual canonical basis via Okounkov type valuations as in Example
\ref{ex-flag-val} and recovers this way the multiplicative property
of the dual canonical basis. Note also that Alexeev and Brion
constructed moduli spaces for spherical varieties and some of their
deformations in \cite{AB}.
See also the survey \cite{brion-hilbert}.

The formula in Corollary \ref{cor-deg} for the degree was first
proved in \cite{brion-nombre-car} and recovered using general
Okounkov bodies techniques in \cite{KK2}. The smoothness criterion
appears in \cite[Th{\'e}or{\`e}me and Corollaire
4.1]{brion-lissite}.

The comparison between the polytope ring and Schubert calculus for
homogeneous spaces of type $A$ have been pursued in
\cite{kirichenko} and \cite{valentina-evgeny}.

The connection between spherical varieties and symplectic geometry
is illustrated by the solution of Delzant's conjecture. In
\cite{delzant}, Delzant conjectured that, if $K$ is a compact Lie
group, the symplectomorphic class of a multiplicity-free Hamiltonian
$K$-manifold (multiplicity-free means here that $K$ has a dense
orbit in the preimage of any coadjoint orbit by the moment map) is
determined by the convex polytope $\mu(X)\cap\mathfrak{t}_{K,+}^\vee$
(obtained as the
intersection of the image of the moment map with the dominant
chamber) and by its general stabiliser (the stabiliser of a general
point in $\mu^{-1}(\mu(X)\cap\mathfrak{t}_{K,+}^\vee)$). Knop
\cite{knop-jams2} proved that this conjecture is equivalent to the
fact that affine spherical varieties are determined by their weight
monoid. This in turn was proved by Losev
\cite{losev2},\cite{losev4}. Furthermore Knop describes in
\cite{knop-jams2}, building on a classification of smooth affine
spherical varieties in \cite{knop-van}, which moment polytopes and
general stabilisers actually occur. We refer to \cite{delzant2},
\cite{losev3} and \cite{knop-jams2} for more details on this topic.

%%%%%%%%%%%%%%%%%%%%%%%%%%%%%%%%%%%%%%%%%%%%%%%%%%%%%%%%%%%%%%%%%%%%%%%%%

\subsection{$B$-orbits}
\label{sec-B-orb}

An interesting combinatorial object associated to any spherical
variety $X$ is the set $B(X)$ of $B$-orbit closures in $X$. As for
homogeneous spaces, this set has two natural orderings: Bruhat and
weak Bruhat orders.

Bruhat order is given by inclusion.
The weak Bruhat order on $B(X)$ is defined as follows. The maximal
elements for the weak order are the closures of the $G$-orbits. If
$Y\in B(X)$ is not $G$-stable, then there exists a minimal parabolic
subgroup $P$ containing $B$ such that $PY\neq Y$. We say that $P$
\emph{raises} $Y$ to $PY$ and write $Y\preccurlyeq PY$. These are
the covering relations of the weak Bruhat order.

\subsubsection{Graph $\Gamma(X)$}
For a covering relation $Y\preccurlyeq PY$, we may consider the
proper morphism $\pi:P\times^BY\to PY$.

\begin{prop}
\label{prop-UNT}
One of the following occurs.
\begin{itemize}
\item Type $U$: $PY=Y\cup Y'$ and $\pi$ is birational.
\item Type $N$: $PY=Y\cup Y'$ and $\pi$ has degree 2.
\item Type $T$: $PY=Y\cup Y'\cup Y''$ for $Y''\in B(X)$ with $\dim
  Y''=\dim Y$ and $\pi$ is birational.
\end{itemize}
\end{prop}

\begin{proof}
For $x$ general in $Y$, write $P_x$ for the stabiliser of $x$ in
$P$. If $\cO$ is the dense $B$-orbit in $Y$, the orbits in $P\cO$
are in bijections with the orbits of $P_x$ in $P/B$. Considering the
image of $P_x$ in $\aut(P/B)$ there are three cases: the image
contains the maximal unipotent subgroup (type $U$), the image is the
normaliser of a maximal torus (type $N$) or the image is a maximal
torus (type $T$).
\end{proof}

\begin{defn}
The graph $\Gamma(X)$ has $B(X)$ as set of vertices and has a
$P$-edge of type $U$, $T$ or $N$ for each covering relation
$Y\preccurlyeq PY$ of this type.
\end{defn}

\begin{example}
For projective rational homogeneous spaces, the graph $\Gamma(X)$ is
connected and has only edges of type $U$.
\end{example}

Some geometric properties of the orbit closures can be seen on the
graph $\Gamma(X)$.

\begin{cor}
\label{coro-non-normal} Assume that $P_1$ and $P_2$ raise $Y$ to
$Y_1$ and $Y_2$ with type $U$ or $T$ and $N$ respectively and that
$P_2$ raises $Y_1$ to $Y_3$ with type $U$ or $T$, then $Y_3$ is not
normal along $Y_2$.
\end{cor}

\vskip -0.8 cm \hskip 1 cm %\begin{multicols}{2}

\centerline{ \begin{tikzpicture}
   \node (Y3) at (9,8)  {$Y_3$};
  \node (Y1) at (9,7) {$Y_1$};
  \node (Y2) at (11,7)  {$Y_2$};
  \node (Y) at (10,6)  {$Y$};
    \draw[double] (Y) -- node {\tiny{2}} (Y2);
    \draw (Y1) -- node {\tiny{2}} (Y3);
    \draw (Y) -- node {\tiny{1}} (Y1);
\end{tikzpicture}}
 %\end{multicols}

\centerline{Graph $\Gamma(X)$}

\centerline{\begin{tikzpicture}
  \node (0) at (2,8.2) {};
  \node (1) at (3,7.5)  {};
  \node (2) at (4,7.5) {};
  \node[anchor=west] (3) at (4,7.5)  {\small{$P_1$-edges of type $U$ or $T$}};
\node (4) at (3,7)  {};
  \node (5) at (4,7) {};
  \node[anchor=west] (6) at (4,7)  {\small{$P_2$-edges of type $U$ or $T$}};
  \node (7) at (3,6.5)  {};
  \node (8) at (4,6.5) {};
  \node[anchor=west] (9) at (4,6.5)  {\small{$P_2$-edges of type $N$}};
    \draw[double] (7) -- node {\tiny{2}} (8);
    \draw (4) -- node {\tiny{2}} (5);
    \draw (1) -- node {\tiny{1}} (2);
\end{tikzpicture}}

% \vskip -0.5 cm

\medskip

\begin{proof}
The morphism $P_2\times^BY_1\to Y_3$ is birational while its
restriction $P_2\times^BY\to Y_2$ has non connected fibres. Zariski
main Theorem gives the conclusion.
\end{proof}

\begin{example}
\label{exam-non-normal} Take $G={\rm Sp}_4$ and $H$ the normaliser
of the Levi subgroup of the maximal parabolic subgroup associated to
the simple long root. Then $G/H$ contains a configuration of
$B$-orbits as in the above corollary (see \cite[Example
1]{brion-com-math-helv}).
\end{example}

\subsubsection{Normality criterion}

As Corollary \ref{coro-non-normal} shows, in some cases, the
existence of edges of type $N$ is the graph $\Gamma(X)$ prevents some
orbit closures from being normal. If there is no edge of type $N$,
Brion proved the following general result.

\begin{defn}
A variety $Y\in B(X)$ is called multiplicity-free if there is no
edge of type $N$ in the full subgraph of $\Gamma(X)$ with vertices
the elements $Y'\succcurlyeq Y$.
\end{defn}

\begin{thm}
\label{thm-normal} Let $Y\in B(X)$ be multiplicity-free and assume
that for $Y'\succcurlyeq Y$, the variety $Y'$ contains a $G$-orbit
if and only if $Y'=GY$. If $GY$ is normal, Cohen-Macaulay or has a
rational resolution, then so does $Y$.
\end{thm}

\begin{proof}
We only sketch the proof of the normality and refer to
\cite{brion-multi-free} for the rest of the proof. We shall need the
following two results.

For $Y\in B(X)$ and $P$ raising $Y$ write $\pi:P\times^B Y\to PY$.
For any sheaf $\cF$ on $P\times^BY$, we have $R^i\pi_*\cF=0$ for
$i>1$ and $R^1\pi_*\cO_{P\times^B Y}=0$ (see for example \cite[Page
294]{brion-com-math-helv}).

For $\cG$ a sheaf on $Y$, write $P\times^B\cG$ the corresponding
sheaf induced on $P\times^BY$. Then if $\cG$ is invariant under the
stabiliser of $Y$ in $G$ and if $\pi_*(P\times^B\cF)=0$ for any $P$
raising $Y$, then $\Supp\cG$ is $G$-invariant (see for example
\cite[Lemma 8]{brion-com-math-helv}).

Assume that $G\cdot Y$ is normal. We prove by induction on the
codimension in $GY$ that $Y$ is normal. Let $\nu:Z\to Y$ be the
normalization and let $\cF$ be the cokernel of the map
$\cO_Y\to\nu_*\cO_Z$. The sheaf $\cF$ is invariant under stabiliser
of $Y$ in $G$. The previous assertions imply that we have an exact
sequence
$$0\to\pi_*\cO_{P\times^B Y}\to\pi_*(P\times^B \nu)_*\cO_{P\times^B Z}
\to\pi_*(P\times^B\cF) \to0$$ if $P$ raises $Y$. But then $P\cdot Y$
is normal and both morphisms $\pi$, $\pi\circ (P\times^B\nu)$ are
proper and birational. Thus, by Zariski's main Theorem the first two
terms are isomorphic and $\pi_*(P\times^B\cF)=0$. We conclude that
$\supp\cF$ is $G$-invariant thus $\cF=0$ since $Y$ contains no
$G$-orbit.
\end{proof}

\begin{remark}
In characteristic $0$, the variety $GY$ is always normal,
Cohen-Macaulay and has a rational resolution.
\end{remark}

The above result and proof is valid for any $G$-variety $X$, see
\cite{brion-multi-free}. This has a very nice application to
subvarieties of homogeneous spaces that we explain in the next
subsection.

\subsubsection{Multiplicity-free classes}
\label{sec-mult}

A natural geometric question for any (smooth) variety $X$ and a
homology class $\a\in H_*(X,\Z)$ is to ask whether this class can be
represented by a subvariety $Y$ of $X$ \emph{i.e.} with $\a=[Y]$ (or
more generally by a linear combination of homology classes of
subvarieties). This is in general a very hard problem since the map
$A_*(X)\to H_*(X)$ is not surjective. This does not occur for smooth
spherical varieties, see Corollary \ref{cor-chow-lisse}. The
following are also classical questions: which cohomology classes are
represented by smooth subvarieties? More generally, given a
cohomology class, what can we say on the varieties representing this
class?

We focus now on the case where $X$ is a projective rational
homogeneous space. For $\a\in H_*(X,\Z)$ a Schubert class, the
problem of determining the (irreducible or smooth) subvarieties $Y$
with $[Y]=\a$ is very classical (see \cite[Section 5.17]{borel-h},
\cite{hironaka}, \cite{hart}). We only have partial answers. For
curves this is completely solved in \cite{perrin-aif}. Very precise
answers are given in the case of compact Hermitian spaces in
\cite{bryant}, \cite{hong}, \cite{coskun}, \cite{robles-the} and
\cite{robles}. In particular, for many Schubert classes, the
varieties representing these classes have to be singular.

More generally, it is natural to ask wether the cohomology class
$[Y]$ of a subvariety $Y$ of $X$ imposes some geometric conditions
on the variety. For examples, codimension conditions on $Y$ or
intersections properties of $[Y]$ may impose geometric conditions on
$Y$.  (in particular simple connectedness results and results on the
Picard group, see \cite{arrondo-caravantes}, \cite{barth},
\cite{barth-larsen}, \cite{debarre}, \cite{faltings},
\cite{fulton-lazarsfeld}, \cite{hartshorne}, \cite{cras},
\cite{indag}, \cite{sommese-vandeven}.

The normality result of Theorem \ref{thm-normal} implies the
following striking result. A cohomology class $\a\in H^*(X,\Z)$ is
called mulitplicity free if it is a non negative linear combination
of Schubert classes with coefficients at most $1$.

\begin{thm}
\label{thm-mfre} Any subvariety $Y$ in $X$ whose cohomology class is
multiplicity-free is normal and Cohen-Macaulay. Furthermore for any
nef line bundle $\mathcal{L}$ on $X$ the map $H^0(X,\mathcal{L})\to
H^0(Y,\mathcal{L})$ is surjective and $H^i(X,\mathcal{L})=0$ for
$i>0$.
\end{thm}

\begin{proof}
We deal with the case $X=G/B$ or more precisely $B\backslash G$ the
quotient by the action by left multiplication. The general case is
similar. Consider the quotient map $p:G\to X$, the variety
$p^{-1}(Y)$ is $B$-stable (for the action of $B$ on $G$ by left
multiplication). Let $P$ be a minimal parabolic subgroup and
consider $\pi_P:B\backslash G\to P\backslash G$. Then $P$ raises
$p^{-1}(Y)$ if and only if $Y\to\pi_P(Y)$ is generically finite.
Furthermore the fibers of this map identifies with the fibers of the
morphism $P\times^Bp^{-1}(Y)\to p^{-1}(Y)$. Let $d$ be their common
degree and let $\a$ be a Schubert class such that $\a=\pi_P^*\beta$
with $\beta$ a Schubert class in $P\backslash G$. Let $\a'$ be the
unique Schubert class such that ${\pi_P}_*\a'=\beta$. Projection
formula gives
$$\int_X[Y]\cdot\a=d\int_X[\pi_P^{-1}(\pi_P(Y))]\cdot\a'$$
and the multiplicity-free assumption of $[Y]$ implies $d=1$. In
particular $p^{-1}(Y)$ is a multiplicity-free subvariety in the
sense of spherical varieties (and $G$-varieties). Since $G$ has a
unique $G$-orbit, Theorem \ref{thm-mfre} implies that $p^{-1}(Y)$
and thus $Y$ is normal and Cohen-Macaulay.

A similar argument implies the existence of flat deformations to a
reduced Cohen-Macaulay union of Schubert varieties proving the
cohomology statements.
\end{proof}

\subsubsection{Comments}
\label{plus-loin-orbites}

Proposition \ref{prop-UNT} was first proved in \cite{RS1} for
symmetric varieties and extends readily to the general case. Brion
studied the graph $\Gamma(X)$ in great details in
\cite{brion-com-math-helv}. The proof of Theorem \ref{thm-normal} is
taken from \cite{brion-multi-free}.

The geometry of $B$-orbit closures in spherical varieties is far
from being completely understood. It has its origin in the
importance of the geometry of Schubert varieties in projective
rational homogeneous space. A combinatorial description of
$\Gamma(X)$ as well as of the weak and strong orders for symmetric
varieties was achieved in \cite{RS1} and \cite{RS2}.

F. Knop in \cite{K2} defines an action of the Weyl group $W$ of $G$
on $B(X)$ and recovers the little Weyl group (see Subsection
\ref{subsect-swg}) as a subgroup of the stabiliser in $W$ of $X$
seen as an element of $B(X)$. N. Ressayre gives in
\cite{ressayre-knop} invariants characterising the orbits of the
Weyl group action on $\Gamma(X)$. In \cite{ressayre-minimal} he
gives a classification and several equivalent definitions of the
spherical varieties for which $W$ acts transitively on the
$B$-orbits within the same $G$-orbit. These varieties are called
spherical varieties of minimal rank and are also characterised by
the fact that $\Gamma(X)$ has only edges of type $U$.

In general it is not known when the closures of $B$-orbits are
normal. Normality and stronger results (such as globally
F-regularity) have been proved by He and Thomsen \cite{HT1} for
closures of $B\times B$-orbits in equivariant compactifications of
reductive groups (see also \cite{HT2} for more results). However,
even for symmetric varieties the closures of $B$-orbits are not
normal in general (see Example \ref{exam-non-normal}). The normality of
$B$-orbit closures is proved for some products of projective
rational homogeneous spaces for a simply-laced semisimple group $G$
which are spherical in \cite{piotr}.

%%%%%%%%%%%%%%%%%%%%%%%%%%%%%%%%%%%%%%%%%%%%%%%%%%%%%%%%%%%%%%%%%%%%%%%%%%%%%%%

\addtocontents{toc}{\protect\setcounter{tocdepth}{1}}
\section{Examples} \label{section-example}

In this section we give some examples of spherical varieties.

\subsection{Toric varieties} Recall that a toric variety is a
normal variety $X$ with a dense orbit of a torus $T$. Toric
varieties are spherical.

\subsection{Projective rational homogeneous spaces} For $G$ reductive
and $P$ a parabolic subgroup, the quotient $G/P$ is a projective
rational homogeneous space. Bruhat decomposition implies that they
are spherical varieties.

\subsection{Horospherical varieties} Horospherical varieties form
the simplest class of spherical varieties containing both toric
varieties and projective rational homogeneous spaces. Horospherical
varieties are the embeddings of homogeneous spaces $G/H$ such that
$H$ contains a maximal unipotent subgroup $U$ of $G$. Bruhat
decomposition implies their sphericity. Another characterisation of
horospherical varieties is the equality $\mathcal{V}=\X^\vee$ (see
for example \cite[Corollary 6.2]{knop}). This simplifies the
combinatorics involved in the classification by colored cones. For
example, the toroidal horospherical varieties are the locally
trivial fibrations $G\times^PY$ over a projective rational
homogeneous space $G/P$ with fiber a variety $Y$ toric for a
quotient of a maximal torus $T$ of $P$ (see \cite{boris-these} for
more precise results on horospherical varieties).

\subsection{Spherical representations} Another very natural class
of examples can be obtained by looking at rational representations
of the group $G$. In particular, a complete classification of
spherical representations has been obtained in \cite{kac} (for
irreducible representations) and in \cite{leahy} and
\cite{benson-rat} in general.

\subsection{Products of homogeneous spaces} It is a classical
problem to ask which product of projective rational homogeneous
spaces $\prod_iG/P_i$ has a dense $G$-orbit. This is solved in
\cite{popov2} if all the parabolic subgroup agree and for $G$ of
type $A$ or $C$ in \cite{MWZ} and \cite{MWZ2}. The more restricted
problem of describing the products of projective rational
homogeneous spaces $\prod_iG/P_i$ which are spherical is solved.
Indeed, using the Structure Theorem \ref{theo-loc-str} this problem
reduces to the classification of spherical representations. The case
of maximal parabolic subgroups was solved in \cite{litt-cone}, the
general case can be found in \cite[Corollaries 1.3.A-G]{stembridge}.
Note that the case of complexity $1$ has recently been solved in
\cite{ponomareva}.

\addtocontents{toc}{\protect\setcounter{tocdepth}{2}}

\providecommand{\bysame}{\leavevmode\hbox to3em{\hrulefill}\thinspace}
\providecommand{\MR}{\relax\ifhmode\unskip\space\fi MR }
\providecommand{\MRhref}[2]{%
  \href{http://www.ams.org/mathscinet-getitem?mr=#1}{#2}
}
\providecommand{\href}[2]{#2}

\end{document}